\documentclass[12pt]{amsart}

\usepackage{amsmath,amssymb,amsthm}
\usepackage[all]{xy}
\usepackage[top=30truemm,bottom=30truemm,left=25truemm,right=25truemm]{geometry}

\newcommand{\Z}{\mathbb{Z}}
\newcommand{\Q}{\mathbb{Q}}
\newcommand{\F}{\mathbb{F}}
\newcommand{\C}{\mathbb{C}}
\newcommand{\pe}{\mathfrak{p}}
\newcommand{\Pe}{\mathfrak{P}}
\newcommand{\qu}{\mathfrak{q}}

\newcommand{\OO}{\mathcal{O}}
\newcommand{\EE}{\mathcal{E}}
\newcommand{\UU}{\mathcal{U}}
\newcommand{\MM}{\mathcal{M}}
\newcommand{\FF}{\mathcal{F}}

\newcommand{\cS}{\mathfrak{S}}
\newcommand{\II}{\mathcal{I}}

\newcommand{\kkappa}{\langle \kappa \rangle}

\newcommand{\LLL}{L}

\newcommand{\dual}{\vee}

\newcommand{\N}{\mathcal{N}}
\newcommand{\NN}{\mathfrak{N}}

\newcommand{\G}{\mathcal{G}}

\newcommand{\CC}{\mathcal{C}}
\newcommand{\PP}{\mathcal{P}}

\newcommand{\QQ}{\mathcal{Q}}

\newcommand{\SF}[1]{\mathcal{F}^{\langle #1 \rangle}}
\newcommand{\varSF}[1]{\mathcal{F}^{[#1]}}

\newcommand{\rpd}{{\rm pd}^{\star}}
\newcommand{\pd}{{\rm pd}}

\DeclareMathOperator{\Gal}{Gal}
\DeclareMathOperator{\GL}{GL}
\DeclareMathOperator{\SL}{SL}
\DeclareMathOperator{\nr}{Nrd}
\DeclareMathOperator{\cyc}{cyc}
\DeclareMathOperator{\Ker}{Ker}
\DeclareMathOperator{\Image}{Im}
\DeclareMathOperator{\Cok}{Cok}
\DeclareMathOperator{\Ext}{Ext}

\DeclareMathOperator{\Ann}{Ann}

\DeclareMathOperator{\ch}{char}

\DeclareMathOperator{\Irr}{Irr}
\DeclareMathOperator{\aug}{aug}
\DeclareMathOperator{\tor}{tor}

\DeclareMathOperator{\Fitt}{Fitt}
\DeclareMathOperator{\Hom}{Hom}
\DeclareMathOperator{\Tor}{Tor}

\DeclareMathOperator{\Left}{left}
\DeclareMathOperator{\Right}{right}
\DeclareMathOperator{\Spec}{Spec}
\DeclareMathOperator{\cd}{cd}
\DeclareMathOperator{\id}{id}
\DeclareMathOperator{\Ind}{Ind}

\let\oldenumerate\enumerate
\renewcommand{\enumerate}{
   \oldenumerate
   \setlength{\itemsep}{1pt}
   \setlength{\parskip}{0pt}
   \setlength{\parsep}{0pt}
}
\let\olditemize\itemize
\renewcommand{\itemize}{
   \olditemize
   \setlength{\itemsep}{1pt}
   \setlength{\parskip}{0pt}
   \setlength{\parsep}{0pt}
}

\theoremstyle{plain}
\newtheorem{thm}{Theorem}[section]
\newtheorem{lem}[thm]{Lemma}

\newtheorem{prop}[thm]{Proposition}
\newtheorem{cor}[thm]{Corollary}

\newtheorem{ass}[thm]{Assumption}
\theoremstyle{definition}
\newtheorem{defn}[thm]{Definition}
\newtheorem{rem}[thm]{Remark}
\newtheorem{eg}[thm]{Example}

\title{Fitting Invariants in Equivariant Iwasawa Theory}
\author{Takenori Kataoka}
\address{Faculty of Science and Technology, Keio University.
3-14-1 Hiyoshi, Kohoku-ku, Yokohama, Kanagawa 223-8522, Japan}
\email{tkataoka@math.keio.ac.jp}
\date{}

\begin{document}

\begin{abstract}
The main conjectures in Iwasawa theory predict the relationship between the Iwasawa modules and the $p$-adic $L$-functions.
Using a certain proved formulation of the main conjecture, Greither and Kurihara described explicitly the (initial) Fitting ideals of the Iwasawa modules for the cyclotomic $\Z_p$-extensions of finite abelian extensions of totally real fields.
In this paper, we generalize the algebraic theory behind their work by developing the theory of ``shifts of Fitting invariants.''
As applications to Iwasawa theory, we obtain a noncommutative version and a two-variable version of the work of Greither and Kurihara.
\end{abstract}

\maketitle

\section{Introduction}\label{sec01}

The main concern of this paper is, in various situations of equivariant Iwasawa theory, how to compute the Fitting ideals (or more generally ``Fitting invariants'') of Iwasawa modules via main conjectures.
We start with reviewing a work of Greither and Kurihara \cite{GK15} which we generalize in this paper.

\subsection*{Review of a result of Greither and Kurihara}
Let $p$ be an odd prime number and $K$ a totally real number field.
We denote by $K^{\cyc}$ the cyclotomic $\Z_p$-extension of $K$.
Let $K'$ be a totally real finite abelian $p$-extension of $K$ which is linearly disjoint with $K^{\cyc}$ over $K$, namely $K' \cap K^{\cyc} = K$ (the assumption that $K'/K$ is a $p$-extension is not essential; see \cite[Note after Theorem 3.3]{GK15}). 
Put $L = K' K^{\cyc}$, which is the cyclotomic $\Z_p$-extension of $K'$.
Then the Galois group $G = \Gal(L/K)$ is the direct product of $\Gal(L/K') \simeq \Z_p$ and the finite abelian group $\Delta = \Gal(L/K^{\cyc})$.
Put $R = \Z_p[[G]]$.
Choosing a topological generator $\gamma$ of $\Gal(L/K')$, we have an isomorphism $R \simeq \Z_p[[T]][\Delta]$ which sends $\gamma$ to $1+T$.
Let $\Sigma$ be a finite set of places of $K$ containing the primes above $p$ and primes which are ramified in the extension $K'/K$.

As an algebraic object, we consider the $\Sigma$-ramified Iwasawa module $X_{\Sigma}(L)$, which is defined as the Galois group of the maximal abelian pro-$p$ extension of $L$ which is unramified outside $\Sigma$.
It is known that $X_{\Sigma}(L)$ is a finitely generated torsion $R$-module.
Here the torsionness of an  $R$-module means that it is torsion as a $\Z_p[[T]]$-module.
On the other hand, as an analytic object, we have the equivariant $p$-adic $L$-function $\Theta_{L/K, \Sigma} \in Q(R)^{\times}$, where $Q(R)$ denotes the ring of fractions of $R$ (see, for instance, \cite[Section 3]{GK15} or \cite[Definition 5.16]{GP15}).
The main conjectures claim that the structure of $X_{\Sigma}(L)$ is closely related to $\Theta_{L/K, \Sigma}$.
In our abelian situation, there are several formulations of main conjectures and they are proved under the hypothesis that the $\mu$-invariant of $X_{\Sigma}(L)$ vanishes (\cite[Theorem 11]{RW02}, \cite[Theorem 5.6]{GP15}).

One of the main objectives in \cite{GK15} is to compute the Fitting ideal of $X_{\Sigma}(L)$ as an $R$-module.
Here we recall the definition of the Fitting ideals (see \cite[Section 3.1]{Nor76} or \cite[Section 20.2]{Eis95}).
For a finitely generated $R$-module $M$, take a presentation 
\[
R^b \overset{h}{\to} R^a \to M \to 0
\]
 with $h$ regarded as an $a \times b$ matrix over $R$.
Define the Fitting ideal of $M$ by
\[
\Fitt_R(M) = (\det(H) \mid \text{$H$ is an $a \times a$ submatrix of $h$})_R,
\]
which does not depend on the choice of the presentation $h$.
If $a > b$, then we have $\Fitt_R(M) = 0$.
To ease the notation, we write $\FF(M) = \Fitt_R(M)$ for a while.
The following observation will be important.
Let $P$ be a finitely generated torsion $R$-module with $\pd_R(P) \leq 1$, where $\pd_R$ denotes the projective dimension over $R$.
Then $\FF(P)$ is principal and invertible as a fractional ideal of $R$ since $P$ fits into an exact sequence of the form $0 \to R^a \to R^a \to P \to 0$ (see also Subsection \ref{subsec41}).

Now, by the philosophy of main conjecture, our object $\FF(X_{\Sigma}(L))$ should be related to the principal ideal $(\Theta_{L/K, \Sigma})$.
But the problem is that the various existing formulations of the main conjectures do not directly tell us the precise description of $\FF(X_{\Sigma}(L))$.

Greither and Kurihara \cite{GK15} overcame such difficulties.
In order to state their result, fix a decomposition $\Delta = \prod_{i=1}^s \Delta^{(i)}$ where $\Delta^{(i)}$ is a cyclic $p$-group.
Then we can construct an exact sequence
\[
\cdots \to \Z_p[\Delta]^{\frac{s(s+1)(s+2)}{6}} \overset{d_3}{\to} \Z_p[\Delta]^{\frac{s(s+1)}{2}} \overset{d_2}{\to} \Z_p[\Delta]^s \overset{d_1}{\to} \Z_p[\Delta] \to \Z_p \to 0
\]
(see Subsection \ref{subseca1}).
Let $B_{\Delta}$ be the cokernel of $d_3$, so we have an exact sequence
\begin{equation}\label{eq01}
0 \to B_{\Delta} \to \Z_p[\Delta]^s \overset{d_1}{\to} \Z_p[\Delta] \to \Z_p \to 0.
\end{equation}
Letting $T$ act trivially, we regard this sequence as a sequence of $R$-modules.

\begin{thm}[{\cite[Theorem 3.3(a)]{GK15}, \cite[Theorem 4.1]{GK17}}]\label{thm01}
Suppose that the $\mu$-invariant of $X_{\Sigma}(L)$ vanishes.
 Then we have
 \[
 \FF(X_{\Sigma}(L)) = \Theta_{L/K, \Sigma} T^{1-s} \FF(B_{\Delta}).
\]
\end{thm}

Here, we already know that $B_{\Delta}$ has the presentation $d_3$ over $\Z_p[\Delta]$.
So the Fitting ideal of $B_{\Delta}$ over $R$ can be computed, though it is complicated when $s$ is large.
In fact, one of the topics in \cite{GK15} is to compute $\FF(B_{\Delta})$, and in \cite{GKT} a set of explicit generators is obtained.
%In fact, one of the topics in \cite{GK15} is to compute $\FF(B_{\Delta})$, and it is stated that a set of explicit generators was obtained.

The main ingredient of the proof of Theorem \ref{thm01} is the existence of an exact sequence
\begin{equation}\label{eq02}
0 \to X_{\Sigma}(L) \to P_1 \to P_2 \to \Z_p \to 0
\end{equation}
of finitely generated torsion $R$-modules with $\pd_R(P_i) \leq 1$ for $i = 1, 2$.
Moreover $P_2$ can be chosen as $P_2 \simeq \Z_p[\Delta]$.
The existence of such a sequence (reproduced in Proposition \ref{prop7a}) is used to formulate the equivariant main conjecture by Ritter and Weiss \cite{RW02}.
Then basically Theorem \ref{thm01} is proved by comparing the sequences (\ref{eq01}) and (\ref{eq02}), though the method is technical and complicated.
Note also that \cite{GK15} uses the character-wise main conjecture of Wiles \cite{Wil90} and does not depend on the result of Ritter and Weiss \cite{RW02}.

\subsection*{Algebraic Results: Shifts of Fitting Invariants}

Keep the above notations.
The first result in this paper clarifies the algebraic theory behind Theorem \ref{thm01} as follows.

\begin{thm}\label{thm02}
Let $M$ be a finitely generated torsion $R$-module.
For any $n \geq 0$, take an exact sequence
\begin{equation}\label{seq:3}
0 \to N \to P_1 \to \dots \to P_n \to M \to 0
\end{equation}
of finitely generated torsion $R$-modules where $\pd_R(P_i) \leq 1$ for $1 \leq i \leq n$.
Then the fractional ideal
 \[
 \varSF{n}(M) = \left(\prod_{i=1}^n \FF(P_i)^{(-1)^i}\right) \FF(N)
 \]
of $R$ is independent of the choice of the sequence (\ref{seq:3}).
\end{thm}

We call $\varSF{n}$ the {\it $n$-th shift} of $\FF$.
The proof of Theorem \ref{thm02} is given by Theorem \ref{thm41new} and Proposition \ref{prop4v}.
In fact, Theorem \ref{thm02} is valid for any commutative noetherian ring $R$ (with the definition of {\it torsion modules} clarified).

We can easily deduce Theorem \ref{thm01} from Theorem \ref{thm02} as follows.
Applying Theorem \ref{thm02} with $n = 2$ to the sequence (\ref{eq02}) yields $\FF(X_{\Sigma}(L)) = \FF(P_1)\FF(P_2)^{-1}\varSF{2}(\Z_p)$.
More impressively, this equality can be written as
\[
\FF(X_{\Sigma}(L)) = \FF(\Phi_{L/K, \Sigma})\varSF{2}(\Z_p)
\]
where $\Phi_{L/K, \Sigma} = [P_1] - [P_2]$ is regarded as an element of the Grothendieck group of the exact category of finitely generated torsion $R$-modules of projective dimension $ \leq 1$. ($\FF$ actually factors through that Grothendieck group: see Remark \ref{rem2a}(3).)
Since $\FF(\Z_p[\Delta]) = (T)$, Theorem \ref{thm02} applied to the sequence (\ref{eq01}) yields $\varSF{2}(\Z_p) = T^{1-s}\FF(B_{\Delta})$.
Finally, Ritter and Weiss \cite[Theorem 11]{RW02} proved the equivariant main conjecture $\FF(\Phi_{L/K, \Sigma}) = (\Theta_{L/K, \Sigma})$ assuming the vanishing of the $\mu$-invariant.
These prove Theorem \ref{thm01}.

In this paper, we further generalize the concept of the shifts of $\FF$ in several directions.
One direction is to consider a noncommutative ring $R$.
In that case, an analogue of the Fitting ideals is not defined in general.
However, for certain rings $R$, such as $R = \Z_p[[T]][\Delta]$ with $\Delta$ a noncommutative finite group, Nickel \cite{Nic10} defined the noncommutative Fitting invariant $\Fitt_R^{\max}$ (we recall the definition in Subsection \ref{subsec42}).
In order to treat the commutative case and the noncommutative case simultaneously, we will introduce an axiomatic Fitting invariant (called simply a {\it Fitting invariant}).
The usual Fitting ideals and the noncommutative Fitting invariants of Nickel are certainly examples of Fitting invariants as we will show in Propositions \ref{prop4v} and \ref{prop4g}, respectively.
Then Theorem \ref{thm41new} asserts that the shifts of any Fitting invariant are well-defined by the same formula as Theorem \ref{thm02}.

Another direction of the generalizations of the theory of shifts is to define the $n$-th shift for a negative integer $n$.
To that end, we will impose Assumption \ref{ass16} and use a duality on a certain category $\CC$ of finitely generated torsion $R$-modules.
We will also have to introduce another slightly different axiomatic Fitting invariant (called a {\it quasi-Fitting invariant}).
A quasi-Fitting invariant is defined only on $\CC$, but satisfies an apparently stronger axiom than a Fitting invariant.
The previous examples of Fitting invariants indeed give rise to examples of quasi-Fitting invariants.
We will also propose a quasi-Fitting invariant for the Iwasawa algebra of a compact $p$-adic Lie group as considered in \cite{CFKSV05}.

For a quasi-Fitting invariant $\FF$, in Theorem \ref{thm47} we will define a well-behaving $n$-th shift $\SF{n}(M)$ for any integer $n$ and any module $M$ in $\CC$.
(The reader can now take a look at the statement of Theorem \ref{thm47} to see the analogy with the definition of $\varSF{n}$.)
Moreover, by Theorem \ref{thm41}, the definition can be extended to any integer $n$ and any finitely generated torsion $R$-module $M$ (not necessarily contained in $\CC$).

As we will see in Subsection \ref{subsec:dual}, one useful application of negative shifts is the simplified proofs (and the generalizations) of the algebraic propositions (\cite[Lemma 5]{BG03}, \cite[Proposition 7.3]{GP15} and \cite[Propositions 5.3.2 and 6.3.2]{Nic10}) claiming the relation between the ``typical'' Fitting invariants of modules in certain four term exact sequences.
From our point of view, their assertions can be written as the simple form $\SF{-2}(M) = \FF(M^*)$ as Propositions \ref{prop4b} and \ref{prop4a}.

\subsection*{Applications to Iwasawa Theory}

In Section \ref{sec5}, we apply our theory of shifts to various situations in Iwasawa theory.
We now sketch the results without explaining the precise notations.
The first application is a non-commutative version of Theorem \ref{thm01}.

\begin{thm}[Theorem \ref{thm74}]\label{thm:13a}
Let $K$ be a totally real number field, $L$ a totally real, Galois extension of $K$ which is a finite extension of $K^{\cyc}$.
Let $\Sigma$ be a finite set of places of $K$ containing all primes above $p$ and primes ramified in $L/K$.
We denote by $\Theta_{L/K, \Sigma}$ the equivariant $p$-adic $L$-function.
Suppose that the $\mu$-invariant of $X_{\Sigma}(L)$ vanishes.
Then for $\FF = \Fitt_R^{\max}$ where $R = \Z_p[[\Gal(L/K)]]$, we have
\[
\FF(X_{\Sigma}(L)) = \Theta_{L/K, \Sigma} \varSF{2}(\Z_p).
\]
\end{thm}
Though Theorem \ref{thm74} uses $\SF{2}(\Z_p)$ instead of $\varSF{2}(\Z_p)$, they are equal by Remark \ref{rem4a} and the fact that $\Z_p \in \CC$.
We also note that $\varSF{2}(\Z_p)$ should be computed by constructing an exact sequence of the form (\ref{eq01}), but the construction seems difficult in general and we do not study the problem in this paper.

The following is a result about a finite CM Galois extension of number fields.
See Remark \ref{rem:59a} for the relation with previous works.

\begin{thm}[Theorem \ref{thm5a}]
Let $K$ be a totally real number field, $K'$ a CM, finite Galois extension of $K$ containing $\mu_p$, the group of $p$-th roots of unity.
Let $\Sigma$ be a finite set of places of $K$ containing all primes above $p$ and primes ramified in $K'/K$.
We denote by $\theta_{K'/K, \Sigma}(s)$ the equivariant zeta function.
Put $L = (K')^{\cyc}$ and $G' = \Gal(L/K')$.
Suppose that the $\mu$-invariant of $X_{\Sigma}(L^+)$ vanishes.
Then for $\FF_{R'} = \Fitt^{\max}_{R'}$ where $R' = \Z_p[\Gal(K'/K)]$ and any $n \in \Z$, $r \geq 2$, we have
\[
\SF{n}_{R'}(X_{\Sigma}(L^+)(-r)_{G'}) = \theta_{K'/K, \Sigma}(1-r)^{(-1)^n} \SF{n+2}_{R'}(\Z_p(-r)_{G'}),
\]
where $(-r)$ denotes the Tate twist.
\end{thm}

The following result has been proved by Greither \cite[Theorem 3]{Grei04}.
In this paper we reprove it via a systematic use of the $(-1)$-st shift $\SF{-1}$.

\begin{thm}[Theorem \ref{thm76} and Lemma \ref{lemzv}]
Consider the situation of Theorem \ref{thm:13a}, and assume that $L/K$ is abelian.
Let $\Sigma_p$ denote the set of primes above $p$ and we consider $X_{\Sigma_p}(L)$.
Let $\psi$ be a non-trivial character of $G$ of order prime to $p$ and the subscript $\psi$ denote the $\psi$-part.
Then for $\FF = \Fitt_R$, we have
\[
\FF(X_{\Sigma_p}(L)_{\psi}) = (\Theta_{L/K, \Sigma})_{\psi} \prod_{v \in \Sigma \setminus \Sigma_p} \left( \frac{\N_{I_v}}{\widetilde{\sigma_v}-\NN(v)}, 1 \right)_{\psi}.
\]
Here $\widetilde{\sigma_v}$ is a lift of the Frobenius element, $\NN(v)$ is the order of the residue field at $v$, $\N_{I_v}$ is the norm element of the group ring associated to the inertia group $I_v$.
\end{thm}

As the final result, we give a two-variable version of Theorem \ref{thm01}.

\begin{thm}[Theorem \ref{thmzx}]
Let $K$ be an imaginary quadratic field in which $p$ splits into two primes $\pe, \overline{\pe}$.
Let $\LLL$ be an abelian extension of $K$ which is a finite extension of $\widetilde{K}$, the $\Z_p^2$-extension of $K$.
Let $\Sigma$ be a finite set of places of $K$ containing all primes above $p$ and primes ramified in $\LLL/K$.
Put $\Sigma_0 = \Sigma \setminus \{\overline{\pe}\}$ and we shall consider $X_{\Sigma_0}(\LLL)$.
Then for $\FF = \Fitt_R$ where $R = \Z_p[[\Gal(\LLL/K)]]$, we have
\[
\FF(X_{\Sigma_0}(\LLL)) = \FF(\Phi_{L/K, \Sigma_0}) \varSF{2}(\Z_p),
\]
where $\Phi_{L/K, \Sigma_0}$ is the $K$-theoretic invariant defined via the sequence in Proposition \ref{prop71}.
\end{thm}
In our forthcoming paper, we hope to study an expected main conjecture relating $\Phi_{L/K, \Sigma_0}$ with a certain $p$-adic $L$-function.

\section*{Acknowledgments}
First I would like to thank my supervisor  Takeshi Tsuji for his continued support during this research.
I would also like to thank Masato Kurihara for several helpful discussions.
Thanks are also due to Andreas Nickel, Cornelius Greither, and the anonymous referees for their helpful comments, which greatly improved this article.
A part of this research was done during my visit to Ralph Greenberg, to whom I am deeply grateful for his great hospitality and for fruitful discussions.
I would like to express my appreciation to the organizers of the conference Iwasawa 2017, where I had the opportunity to give a talk about the contents of this article.
This work was supported by the Program for Leading Graduate Schools (FMSP) at the University of Tokyo, and partially by JSPS KAKENHI Grant Number 17J04650.

\section{Fitting Invariants}

\subsection{Shift}\label{subsec11}
Let $R$ be a (not necessarily commutative) ring with unity, which is (both left and right) noetherian.
Let $S$ be an Ore set of $R$ consisting of non-zero-divisors.
In other words, $S$ is a set of non-zero-divisors of $R$ satisfying the following conditions:
\begin{itemize}
\item $1 \in S$.
\item If $f, g \in S$, then $fg \in S$.
\item If $f \in S, a \in R$, then $fR \cap aS \neq \emptyset$ and $Rf \cap Sa \neq \emptyset$.
\end{itemize}
Note that the final condition is trivial if $R$ is commutative or, more generally, $S$ is central in $R$.
A typical example to keep in mind is the following.

\begin{eg}\label{ega1}
Let $p$ be a prime number.
Let $G$ be a profinite group containing an open central subgroup $G'$ which is isomorphic to $\Z_p^d$ with $d \geq 0$.
Put $R = \Z_p[[G]], \Lambda = \Z_p[[G']]$ and $S = \Lambda \setminus \{0\}$.
Then $S$ is indeed an Ore set since $G'$ is central.
It is known that $\Lambda$ is isomorphic to the ring $\Z_p[[T_1, \dots, T_d]]$ of formal power series in $d$ variables.
\end{eg}

In this section, a module basically means a left module.
However, every argument is also valid for right modules by symmetry.
An $R$-module $M$ is said to be $S$-torsion if every element of $M$ is annihilated by an element of $S$.
Let $\MM = \MM_{R, S}$ be the category of all finitely generated $S$-torsion left $R$-modules.

For a left $R$-module $M$, we denote by $\pd_R(M)$ the projective dimension of $M$ over $R$, which can be defined by
\[
\pd_R(M) = \sup \{i \in \Z \mid \text{$\Ext^i_R(M, N) \neq 0$ for some left $R$-module $N$} \}.
\]
If $\Ext^i_R(M,N) = 0$ for any $N$ and $i$, which happens only when $M = 0$, then we put $\pd_R(M) = -\infty$.
Hence $\pd_R(M)$ is a nonnegative integer or $\pm \infty$.
Now we define a full subcategory $\PP = \PP_{R, S}$ of $\MM$ by
\[
\PP = \{ P \in \MM \mid \pd_R(P) \leq 1 \}.
\]

\begin{lem}\label{lem2a}
Let $0 \to M' \to M \to M'' \to 0$ be an exact sequence of $R$-modules.

(1) $M$ is in $\MM$ if and only if $M'$ and $M''$ are in $\MM$.

(2) Suppose that $M''$ is in $\PP$.
Then $M$ is in $\PP$ if and only if $M'$ is in $\PP$.
\end{lem}

\begin{proof}
(1) Clear.

(2) This can be proved by taking the extension groups of the given short exact sequence. 
\end{proof}

\begin{lem}\label{lem2b}
(1) For any $f \in S$, the $R$-module $P = R/Rf$ is contained in $\PP$.

(2) For any $M \in \MM$, there are a module $P \in \PP$ and a surjective homomorphism $P \to M$.
\end{lem}

\begin{proof}
(1) $P$ is $S$-torsion by the Ore property of $S$.
Then the exact sequence $0 \to R \overset{\times f}{\to} R \to P \to 0$ of left $R$-modules shows that $P \in \PP$.

(2) Let $x_1, \dots, x_r$ be a set of generators of $M$ and choose $f_i \in S$ such that $f_i x_i = 0$.
Then we obtain a surjective map $P = \bigoplus_{i=1}^r R/Rf_i \to M$ induced by $(a_i)_i \mapsto \sum_{i=1}^r a_i x_i$, and $P \in \PP$ follows from (1).
\end{proof}

As mentioned in Section \ref{sec01}, to treat simultaneously the Fitting ideals in the commutative case and the noncommutative Fitting invariants defined by Nickel, we introduce a kind of axiomatic Fitting invariants as follows.

\begin{defn}
A Fitting invariant is a map $\FF: \MM \to \Omega$ where $\Omega$ is a commutative monoid,  satisfying the following conditions:
\begin{itemize}
\item If $P \in \PP$, then $\FF(P) \in \Omega^{\times}$, the group of invertible elements of $\Omega$.
\item If $0 \to M' \to M \to P \to 0$ is an exact sequence in $\MM$ with $P \in \PP$, then $\FF(M) = \FF(P)\FF(M')$.
\end{itemize}
Here, even though $\MM$ is a category, a ``map'' merely means that $\FF$ defines an element $\FF(M)$ of $\Omega$ for each object $M$ of $\MM$.
Since the second condition implies in particular that isomorphic objects are sent to the same element, $\FF$ also can be regarded as an actual map from the set of the isomorphism classes of objects of $\MM$.
In this paper, we will often similarly abuse the word ``map''.
\end{defn}

\begin{rem}\label{rem2a}
(1) We can see that  $\FF(0)$ is the identity element of $\Omega$. (Here $0$ denotes the zero module.)

(2) By a standard argument, it can be shown that there is a unique universal Fitting invariant $\widetilde{\FF}: \MM \to \widetilde{\Omega}$ in a natural sense.
Namely, $\widetilde{\FF}$ is a Fitting invariant such that, for every Fitting invariant $\FF: \MM \to \Omega$, there is a unique monoid homomorphism $\phi: \widetilde{\Omega} \to \Omega$ such that $\FF = \phi \circ \widetilde{\FF}$.
It seems an interesting question to what extent we can distinguish modules by the values of $\widetilde{\FF}$.
See Example \ref{ega2}.

(3) Consider the Grothendieck group $K_0(\PP)$ of the exact category $\PP$.
By the properties of a Fitting invariant, the restriction $\FF|_{\PP}$ of $\FF$ to $\PP$ induces a group homomorphism $K_0(\PP) \to \Omega^{\times}$.
So $\FF$ cannot distinguish modules in $\PP$ which represent the same element in $K_0(\PP)$.
\end{rem}

The following is the fundamental theorem in this paper.

\begin{thm}\label{thm41new}
Let $\FF: \MM \to \Omega$ be a Fitting invariant.
Then for any $n \geq 0$, the following map $\varSF{n}: \MM \to \Omega$ is well-defined.
For each $M \in \MM$, take an exact sequence
\[
0 \to N \to P_1 \to \dots \to P_n \to M \to 0
\]
in $\MM$ with $P_1, \dots, P_n \in \PP$ and define
 \[
 \varSF{n}(M) = \left(\prod_{i=1}^n \FF(P_i)^{(-1)^i}\right) \FF(N).
 \]
 Moreover, $\varSF{n}: \MM \to \Omega$ is again a Fitting invariant.
\end{thm}

Before the proof, we note that $\varSF{n}(P) = \FF(P)^{(-1)^n}$ for any $P \in \PP$, as long as the well-definedness of $\varSF{n}(P)$ is established.
This is shown by taking $N = P_1 = \dots = P_{n-1} = 0, P_n = P$ if $n \geq 1$, and by taking $N = P$ if $n =0$.
We also note that we have defined $\varSF{0}(M) = \FF(M)$ for any $M \in \MM$, and it is obviously well-defined.

\begin{proof}
The existence of such an exact sequence follows from Lemmas \ref{lem2a}(1) and \ref{lem2b}(2).
First we show that it is enough to show the assertion for $n = 1$.
The case $n = 0$ is obvious.
Suppose $n \geq 2$, then by induction we may assume that $\varSF{n-1}$ is well-defined and is a Fitting invariant.
By the case $n = 1$, the Fitting invariant $\left( \varSF{n-1}\right)^{[1]}$ is well-defined.
Then for a sequence defining $\varSF{n}(M)$, letting $L$ be the kernel of $P_n \twoheadrightarrow M$, we have
\begin{align*}
(\varSF{n-1})^{[1]}(M) 
& = \varSF{n-1}(P_n)^{-1} \varSF{n-1}(L) \\
& = \FF(P_n)^{(-1)^n} \left(\prod_{i=1}^{n-1} \FF(P_i)^{(-1)^i}\right) \FF(N) \\
& = \left(\prod_{i=1}^n \FF(P_i)^{(-1)^i}\right) \FF(N).
\end{align*}
This shows $\varSF{n} = \left( \varSF{n-1}\right)^{[1]}$.
Thus the assertion for $n = 1$ implies the whole assertion.

Now let us show the assertion for $n = 1$.
Take two exact sequences $0 \to N \to P \to M \to 0$ and $0 \to N' \to P' \to M \to 0$ in $\MM$ with $P, P' \in \PP$.
Then we can construct a commutative diagram with exact rows and columns
\[
\xymatrix{
0 \ar[r] & N \ar[r] \ar@{=}[d] & P \ar[r] & M \ar[r] & 0 \\
0 \ar[r] & N \ar[r] & L \ar@{->>}[u] \ar[r]  & P' \ar[r] \ar@{->>}[u]& 0 \\
& & N' \ar@{=}[r] \ar@{^{(}->}[u] & N' \ar@{^{(}->}[u]& 
}
\]
where $L$ is defined as the pull-back of $P$ and $P'$ over $M$.
We have $L \in \MM$ and
\[
\FF(P)^{-1} \FF(N) = \FF(P)^{-1} \FF(P')^{-1} \FF(L) = \FF(P')^{-1} \FF(N'),
\]
which shows the well-definedness of $\varSF{1}$.

To show that $\varSF{1}$ is a Fitting invariant, let $0 \to M' \to M \to P \to 0$ be an exact sequence in $\MM$ with $P \in \PP$.
Take an exact sequence $0 \to N \to P' \to M \to 0$ in $\MM$ with $P' \in \PP$.
Then we can construct 
\[
\xymatrix{
0 \ar[r] & M' \ar[r] & M \ar[r] & P \ar[r] & 0 \\
0 \ar[r] & P'' \ar[r] \ar@{->>}[u] & P' \ar@{->>}[u] \ar[r]  & P \ar[r] \ar@{=}[u]& 0 \\
& N \ar@{=}[r] \ar@{^{(}->}[u] & N \ar@{^{(}->}[u] & &
}
\]
where $P''$ is defined as the kernel of $P' \to M \to P$.
Since $P'' \in \PP$, we obtain 
\[
\varSF{1}(M) = \FF(P')^{-1}\FF(N) = \FF(P)^{-1}\FF(P'')^{-1}\FF(N) = \varSF{1}(P)\varSF{1}(M').
\]
This completes the proof.
\end{proof}

\subsection{Fitting Ideals over a Commutative Ring}\label{subsec41}
As the most fundamental example, when $R$ is commutative, we show that the (initial) Fitting ideal can be regarded as a Fitting invariant.
Let $R$ be a noetherian commutative ring and $S$ a multiplicative set of $R$ consisting of non-zero-divisors.
Recall that the Fitting ideal $\Fitt_R(M)$ for a finitely generated $R$-module $M$ is defined as in Section \ref{sec01}.
The definition is valid over any noetherian commutative rings.

\begin{prop}\label{prop4v}
Let $\Omega$ be the commutative monoid of finitely generated $R$-submodules of $S^{-1}R$.
Then the map $\FF:\MM \to \Omega$ defined by $\FF(M) = \Fitt_R(M)$ is a Fitting invariant.
\end{prop}

\begin{proof}
We first show that $\FF(P) \in \Omega^{\times}$ for $P \in \PP$.
Finitely generated $R$-submodules of $S^{-1}R$ will also be called {\it fractional ideals} of $R$ with respect to $S$.
Therefore $\Omega$ is the monoid of fractional ideals of $R$ with respect to $S$.

At first we suppose that $R$ is a local ring.
Then $P \in \PP$ implies that there is an exact sequence $0 \to R^a \overset{h}{\to} R^a \to P \to 0$.
Since $P$ is $S$-torsion, $h$ is an invertible matrix as a matrix over $S^{-1}R$.
Therefore $\Fitt_{R}(P) = (\det(h))$ is a principal ideal generated by an element which is invertible in $S^{-1}R$.
In particular it is invertible as a fractional ideal of $R$ with respect to $S$.

We consider general $R$.
For any prime ideal $\qu$ of $R$, we take the image $S_{\qu}$ of $S$ under the canonical map $R \to R_{\qu}$ as a fixed multiplicative set of $R_{\qu}$.
Then it can be shown that an ideal $I$ of $R$ is invertible as a fractional ideal of $R$ with respect to $S$ if and only if $IR_{\qu}$ is invertible as a fractional ideal of $R_{\qu}$ with respect to $S_{\qu}$ for any prime ideal $\qu$ of $R$.
For any $\qu$, the local case implies that $\Fitt_{R}(P)R_{\qu} = \Fitt_{R_{\qu}}(P_{\qu})$ is invertible as a fractional ideal of $R_{\qu}$ with respect to $S_{\qu}$.
Consequently $\Fitt_R(P)$ is invertible.

Finally, let us confirm the second condition to be a Fitting invariant.
By localization, we may assume that $R$ is a local ring and hence $P$ has a presentation of the form $0 \to R^a \to R^a \to P \to 0$ as above.
Then the condition follows from, for example, \cite[Lemma 3]{CG98}.
\end{proof}

\begin{rem}\label{rem4i}
Relating to Remark \ref{rem2a}, we consider whether $\FF|_{\PP}: K_0(\PP) \to \Omega^{\times}$ is injective for this Fitting invariant $\FF = \Fitt_R$.
We can illustrate the map using the following commutative diagram
\[
\xymatrix{
K_1(R) \ar[r] \ar@{->>}[d]_{\det} & K_1(S^{-1}R) \ar[r]^-{\partial} \ar@{->>}[d]_{\det} & K_0(\PP) \ar[r] \ar[d]_{\FF|_{\PP}} & K_0(R) \ar[r] & K_0(S^{-1}R) \\
R^{\times} \ar@{^{(}->}[r] & (S^{-1}R)^{\times} \ar[r]  & \Omega^{\times} & & 
}
\] 
where the upper row is the localization exact sequence in $K$-theory (see \cite[Theorem III. 3.2]{Wei13}).
The map $(S^{-1}R)^{\times} \to \Omega^{\times}$ is defined by corresponding the principal fractional ideal generated by the element, which makes the lower sequence also exact.

By diagram chasing, if $\partial$ is surjective and $\det: K_1(S^{-1}R) \to (S^{-1}R)^{\times}$ is an isomorphism, then $\FF|_{\PP}: K_0(\PP) \to \Omega^{\times}$ is injective.
We shall show that these two assumptions hold for the situation in Example \ref{ega1} with $G$ commutative.
The fact that $R$ is a direct product of local rings shows the surjectivity of $\partial$ (see also Lemma \ref{lem4d}).
The ring $S^{-1}R$ is a finite extension of the field $S^{-1}\Lambda = Q(\Lambda)$.
Hence $S^{-1}R$ is an artinian ring and consequently a direct product of local rings.
Then \cite[Corollary 2.2.6]{Ros94} implies that $\det: K_1(S^{-1}R) \to (S^{-1}R)^{\times}$ is an isomorphism.
\end{rem}

\begin{eg}\label{ega2}
Let $R = \Z_p[[T]][\Delta], S = \Z_p[[T]] \setminus \{0\}$ with $\Delta$ a finite cyclic group generated by $\delta$ of order $m$.
Consider the $R$-modules $M = \Z_p \oplus \Z_p = R/(\delta - 1, T) \oplus R/(\delta - 1, T)$ and $N = R/(\delta - 1, T)^2$, which are not isomorphic to each other.
For the Fitting invariant $\Fitt_R$, we have
\[
\Fitt_R(M) = (\delta-1, T)^2 = \Fitt_R(N).
\]
Observe that $N$ fits into an exact sequence $0 \to \Z_p/m\Z_p \oplus \Z_p \to N \to \Z_p \to 0$ of $R$-modules, where the first map sends $(1,0), (0,1)$ to $\delta-1, T \in N$.
This is because, putting $\tau = \delta - 1$, 
\[
N \simeq \Z_p[[T]][\tau]/((\tau+1)^m -1, \tau^2, \tau T, T^2) = \Z_p[\tau, T]/(m\tau, \tau^2, \tau T, T^2).
\]

Suppose $p$ divides $m$.
It is clear that $\Fitt_{\Z_p[[T]]}$ also gives a Fitting invariant for $R$.
Since $\pd_{\Z_p[[T]]}(\Z_p) = 1$, we have
\[
\Fitt_{\Z_p[[T]]}(N) = (m, T)T^2 \neq (T^2) = \Fitt_{\Z_p[[T]]}(M)
\]
and in particular the universal Fitting invariant (Remark \ref{rem2a}(2)) distinguishes $M$ and $N$.
 
Suppose $p$ does not divide $m$.
 Then $\Z_p \in \PP_{R, S}$.
 Therefore the exact sequence $0 \to \Z_p \to N \to \Z_p \to 0$ implies $\FF(N) = \FF(M)$ for any Fitting invariant $\FF$.
 Namely, the universal Fitting invariant does not distinguish $M$ and $N$.
\end{eg}

\subsection{Noncommutative Fitting Invariants}\label{subsec42}
First we briefly review the noncommutative Fitting invariant introduced by Nickel \cite{Nic10} and developed by Nickel and Johnston \cite{JN13}.
There are several slightly different formulations, and we will basically follow the survey article \cite[Section 2]{Nic17}.
Let $\Lambda$ be a commutative noetherian complete local domain and $Q(\Lambda)$ its fraction field (see \cite[Remark 2.15]{Nic17}).
Let $A$ be a separable $Q(\Lambda)$-algebra and $R$ a $\Lambda$-order of $A$.

For any positive integer $a$, the matrix ring $M_a(A)$ is also a separable $Q(\Lambda)$-algebra and 
hence one has the reduced norm map 
\[
\nr: M_a(A) \to Z(M_a(A)) = Z(A),
\]
 where $Z(-)$ denotes the center.
Put 
\[
\II(R) = (\nr(H) \mid H \in M_a(R), a \geq 1 )_{Z(R)},
\]
which is a subring of $Z(A)$ containing $Z(R)$.

Let $M$ be a finitely generated left $R$-module.
For a finite presentation 
$R^b \overset{h}{\to} R^a \to M \to 0$ of $M$, identifying $h$ with an $a \times b$ matrix,
define the Fitting invariant of $h$ by
\[
\Fitt_{R}(h) = (\nr(H) \mid \text{$H$ is an $a \times a$ submatrix of $h$})_{\II(R)},
\]
which is an ideal of $\II(R)$.
Define the maximal Fitting invariant $\Fitt^{\max}_{R}(M)$ of $M$ to be the unique maximal Fitting invariant of a finite presentation of $M$ among all the presentations.
If $P$ is a finitely generated $R$-module which admits a quadratic presentation, i.e., a presentation of the form $R^a \overset{h}{\to} R^a \to P \to 0$,
then it is known that $\Fitt^{\max}_{R}(P) = \Fitt_R(h)$ (\cite[Proposition 2.17]{Nic17}).

Now we get back to the main discussion.
Suppose that an Ore set $S$ of $R$ is given so that $S \subset \Lambda \setminus \{0\}$, which is necessarily central in $R$.
Also we need the following.

\begin{ass}\label{ass35}
Any $P \in \PP$ has a quadratic presentation.
\end{ass}

In the situation of Example \ref{ega1}, Assumption \ref{ass35} holds as we will show in Proposition \ref{prop4c}.
But in general Assumption \ref{ass35} can fail to hold (Example \ref{eg4f}).

We denote by $\overline{S}$ the image of $S$ under the map $\nr: R \to \II(R)$.
Then $\overline{S}$ is a multiplicative subset of $\II(R)$ consisting of non-zero-divisors.

\begin{prop}\label{prop4g}
Suppose that Assumption \ref{ass35} holds.
Let $\Omega$ be the commutative monoid of finitely generated $\II(R)$-submodules of $\overline{S}^{-1}\II(R)$.
Then the map $\FF: \MM \to \Omega$ defined by $\FF(M) = \Fitt^{\max}_R(M)$ is a Fitting invariant.
\end{prop}

\begin{proof}
For any $P \in \PP$, there is a presentation $0 \to R^a \overset{h}{\to} R^a \to P \to 0$ by Assumption \ref{ass35}.
Since $h \in M_a(R) \cap \GL_a(S^{-1}R)$, there are $f \in S$ and $h' \in M_a(R) \cap \GL_a(S^{-1}R)$ such that $hh' = h'h = f1_a$.
Then we have $\nr(h)\nr(h') = \nr(f)^a \in \overline{S}$.
Therefore $\FF(P) = (\nr(h))_{\II(R)}$ is invertible in $\Omega$.

Let us confirm the second condition of a Fitting invariant.
Let $0 \to M' \to M \to P \to 0$ be an exact sequence in $\MM$ with $P \in \PP$.
The inclusion $\FF(M) \supset \FF(P)\FF(M')$ is already known (\cite[Proposition 3.5.3]{Nic10}).
In order to show the other inclusion, take a presentation $R^b \overset{h}{\to} R^a \to M \to 0$ such that $\Fitt_R(h) = \Fitt^{\max}_R(M)$.
Construct a commutative diagram with exact rows and columns
 \[
\xymatrix{
0 \ar[r] & M' \ar[r] & M \ar[r] & P \ar[r] & 0 \\
0 \ar[r] & N \ar[r]_{h'} \ar@{->>}[u] & R^a \ar@{->>}[u] \ar[r]  & P \ar[r] \ar@{=}[u]& 0 \\
& R^b \ar@{=}[r] \ar[u]_{h''} & R^b \ar[u]_{h} & &
}
\]
where $N$ is defined as the kernel of $R^a \to M \to P$.
Since $P$ admits a quadratic presentation, $N$ must be a stably free module.
Hence adding a free module to $N, R^a, R^b$ allows us to assume that $N$ is a free $R$-module of rank $a$.
We identify $h, h', h''$ with the representation matrix with respect to a fixed isomorphism $N \simeq R^a$.
Then
\[
\FF(M) = \Fitt_R(h) = \Fitt_R(h'h'') = \nr(h') \Fitt_R(h'') \subset \FF(P)\FF(M').
\]
This completes the proof.
\end{proof}

Let us discuss the validity of Assumption \ref{ass35}.

\begin{lem}\label{lem4d}
Recall the localization exact sequence
\[
K_1(R) \to K_1(S^{-1}R) \overset{\partial}{\to} K_0(\PP) \to K_0(R) \to K_0(S^{-1}R).
\]
Then Assumption \ref{ass35} is equivalent to the surjectivity of the map $\partial$.
\end{lem}

\begin{proof}
If Assumption \ref{ass35} holds, then for any $P \in \PP$, we have a presentation $0 \to R^a \overset{h}{\to} R^a \to P \to 0$.
Since $P$ is $S$-torsion, $h$ is invertible as a matrix over $S^{-1}R$.
Therefore $[P] \in K_0(\PP)$ is the image of $[h] \in K_1(S^{-1}R)$ under $\partial$ and hence $\partial$ is surjective.
The converse can be shown similarly.
\end{proof}

\begin{prop}\label{prop4c}
Consider $R = \Z_p[[G]] \supset \Lambda = \Z_p[[G']], S = \Lambda \setminus \{0\}$ as in Example \ref{ega1}.
Then $A = S^{-1}R = Q(R)$ is actually a separable algebra and Assumption \ref{ass35} holds.
\end{prop}

\begin{proof}
The assertion that $A$ is a separable algebra is proved in \cite[Proposition 5(1)]{RW04} when $d=1$, and the general case can be proved by exactly the same method as follows.
Since $Q(R)$ is a finite dimensional $Q(\Lambda)$-algebra, it is enough to show that $Q(R)$ does not contain any nilpotent left ideals other than zero.
Suppose that $I$ is a nilpotent left ideal of $Q(R)$.
Then $I \cap R$ is also a nilpotent left ideal of $R$.
For each open subgroup $U$ of $G'$, the image $(I \cap R)\Z_p[G/U]$ is a nilpotent left ideal and hence zero, since $\Q_p[G/U]$ does not contain any nilpotent left ideals other than zero.
Therefore $I \cap R$ itself is zero and $I = 0$, as desired.

Now the proof of Assumption \ref{ass35} is a direct generalization of \cite[Lemma 6.2]{Nic10} and \cite[Lemma 13]{RW04}, where $d = 1$.
If $G$ is abelian, then $R$ is a direct product of local rings and Assumption \ref{ass35} follows immediately, as mentioned in Remark \ref{rem4i}.
For general $G$, put $\Delta = G/G'$.
For each cyclic subgroup $\Delta_0$ of $\Delta$, the inverse image $G_0$ of $\Delta_0$ under the map $G \twoheadrightarrow \Delta$ is abelian, so we obtain the commutative diagram
\[
\xymatrix{
K_0(\PP_{\Z_p[[G_0]], S}) \ar[r]_-{0} & K_0(\Z_p[[G_0]]) \ar@{^{(}->}[r] & K_0(\Z_p[\Delta_0]) \ar@{^{(}->}[r] & K_0(\Q_p[\Delta_0]) \\
K_0(\PP_{\Z_p[[G]], S}) \ar[r] \ar[u] & K_0(\Z_p[[G]]) \ar@{^{(}->}[r] \ar[u] & K_0(\Z_p[\Delta]) \ar@{^{(}->}[r] \ar[u] & K_0(\Q_p[\Delta]) \ar[u]
}
\]
Here the $0$ comes from the abelian case with Lemma \ref{lem4d} and the injectivities come from \cite[Proposition 6.20, Theorem 32.1]{CR90}.
Moreover, the map $K_0(\Q_p[\Delta]) \to \prod_{\Delta_0} K_0(\Q_p[\Delta_0])$ is injective, where $\Delta_0$ runs over cyclic subgroups of $\Delta$.
Therefore the map $K_0(\PP_{\Z_p[[G]], S}) \to K_0(\Z_p[[G]])$ is also $0$, which completes the proof.
\end{proof}

\begin{eg}\label{eg4f}
Assumption \ref{ass35} fails for the ring
\[
R = \left\{ \begin{pmatrix} a & b \\ c & d \end{pmatrix} \in M_2(\Z_p) \mid c \in p\Z_p \right\}
\]
and $S = \Z_p \setminus \{0\}$.
More generally, Assumption \ref{ass35} fails if $R$ is a hereditary, non-maximal order over a complete discrete valuation ring.
This fact follows from \cite[Theorem 26.28]{CR90}; the author thanks Andreas Nickel for pointing this out.
\end{eg}

\begin{rem}\label{rem4j}
Suppose that Assumption \ref{ass35} holds.
Similar to Remark \ref{rem4i}, we have the commutative diagram with exact rows
\[
\xymatrix{
K_1(R) \ar[r] \ar[d]_{\nr} & K_1(S^{-1}R) \ar@{->>}[r]^-{\partial} \ar[d]_{\nr} & K_0(\PP) \ar[d]_{\FF|_{\PP}} \\
\II(R)^{\times} \ar@{^{(}->}[r] & (\overline{S}^{-1} \II(R))^{\times} \ar[r]  & \Omega^{\times}
}
\] 
as mentioned in \cite[Remark 2.18]{Nic17}.
Here the surjectivity of $\partial$ comes from Assumption \ref{ass35} and Lemma \ref{lem4d}.
The injectivity of $\FF|_{\PP}$ appears to be a difficult problem.
\end{rem}

\section{Quasi-Fitting Invariants}\label{sec21}

\subsection{Duality}\label{subsec21}
In this subsection, we will exhibit a duality (Proposition \ref{prop23}) on a certain category $\CC$.
Let $R, S$ be as in Subsection \ref{subsec11}, namely, $R$ is a noetherian ring and $S$ is an Ore set of $R$ consisting of non-zero-divisors.
It becomes more important to distinguish whether a module is a left module or a right module in this section than in the previous section, so we will sometimes clarify it.
But by symmetry, every argument about left modules is applicable to right modules, and vice versa.

\begin{defn}
For a left (resp. right) $R$-module $M$ and $i \in \Z$, we put $E^i_R(M) = \Ext^i_R(M, R)$, which is a right (resp. left) $R$-module.
We formally put $E^i_R(M) = 0$ if $i < 0$.
Moreover, put $M^+ = E^0_R(M) = \Hom_R(M,R)$ and $M^* = E^1_R(M)$.
\end{defn}

The functors $E^i_R$ have already played important roles in Iwasawa theory (see \cite{Jan89, Ven02}, for instance).
However, the author did not find a suitable reference for our aim, so we will develop a self-contained theory.

Since $R$ is noetherian, if $M$ is finitely generated then the module $E^i_R(M)$ is also finitely generated for any $i$.
As a variant of $\pd_R(M)$, whose definition is recalled in Subsection \ref{subsec11}, we put
\[
\rpd_R(M) = \sup \{i \in \Z \mid E^i_R(M) \neq 0 \}.
\]
Then clearly $\pd_R(M) \geq \rpd_R(M)$.
Under Assumption \ref{ass16} below, we can show that $\rpd_R(M)=-\infty$ if and only if $M = 0$ (see Remark \ref{remzz}(1)).

\begin{lem}\label{lem11}
Let $M$ be a finitely generated $R$-module.
If $\pd_R(M) < \infty$, then $\pd_R(M) = \rpd_R(M)$.
\end{lem}

\begin{proof}
The assertion is trivial if $\pd_R(M) = -\infty$.
Putting $d = \pd_R(M) \neq \pm \infty$, we shall show that $E^d(M) \neq 0$.
Suppose to the contrary $E^d(M) = 0$.
Take an exact sequence
\[
0 \to F_d \to F_{d-1} \to \dots \to F_0 \to M \to 0
\]
of finitely generated $R$-modules with $F_0, \dots, F_d$ projective.
The assumption $E^d(M) = 0$ implies that $\Ext^d_R(M, F_d) = 0$.
On the other hand, the definition of $\Ext$ implies 
\[
\Ext^d_R(M, F_d) \simeq \Cok(\Hom_R(F_{d-1}, F_d) \to \Hom_R(F_d, F_d)).
\]
Therefore we obtain a splitting of the injective map $F_d \to F_{d-1}$ and we can construct a projective resolution of $M$ of length $d-1$, which contradicts $d = \pd_R(M)$.
This proves the lemma.
\end{proof}

The following assumption will be crucial.

\begin{ass}\label{ass16}
There is an integer $d(R)$ such that $\rpd_R(M) \leq d(R)$ for any finitely generated (left or right) $R$-module $M$. 
\end{ass}

The next Proposition \ref{prop18} will show that the existence of a certain subring $\Lambda$ of $R$ is sufficient for Assumption \ref{ass16} to hold.
In fact, the existence of $\Lambda$ holds in all example in this paper.
However, since the choice of such a ring $\Lambda$ is non-canonical, the statement of Assumption \ref{ass16} has the advantage that it depends only on the ring structure of $R$.

\begin{prop}\label{prop18}
Suppose that $R$ contains a subring $\Lambda$ such that 
\begin{itemize}
\item $R$ is finitely generated and projective as a left $\Lambda$-module and as a right $\Lambda$-module.
\item There are isomorphisms $\Hom_\Lambda(R, \Lambda) \simeq R$ as $(R, \Lambda)$-bimodules and as $(\Lambda, R)$-bimodules.
\item The global dimension of $\Lambda$ is finite, namely, there is an integer $d(\Lambda)$ such that $\pd_{\Lambda}(M) \leq d(\Lambda)$ for any finitely generated (left or right) module $M$ over $\Lambda$.
\end{itemize}
Then the following are true.

(1) For any left (resp. right) $R$-module $M$ and $i \geq 0$, we have 
\[
E^i_R(M) \simeq E^i_{\Lambda}(M)
\]
 as right (resp. left) $\Lambda$-modules.

(2) $R$ satisfies Assumption \ref{ass16}.
\end{prop}

\begin{proof}
(1) For any (left or right) $R$-module $M$, we have functorial isomorphisms
\[
\Hom_R(M, R) \simeq \Hom_R(M, \Hom_{\Lambda}(R, \Lambda)) \simeq \Hom_{\Lambda}(R \otimes_R M, \Lambda) \simeq \Hom_{\Lambda}(M, \Lambda)
\]
depending on the choice of an isomorphism $\Hom_{\Lambda}(R, \Lambda) \simeq R$.
Since a projective $R$-module is projective over $\Lambda$ by the assumption, the claim $\Ext^i_R(M, R) \simeq \Ext^i_{\Lambda}(M, \Lambda)$ follows from the construction of the Ext functor using a projective resolution of $M$ as an $R$-module.

(2) Putting $d(R) = d(\Lambda)$, we can immediately show Assumption \ref{ass16} by (1).
\end{proof}

Using Proposition \ref{prop18}, we can verify Assumption \ref{ass16} in the situation of Example \ref{ega1}.
More generally, the following is true. (It is essentially \cite[Corollary 2.4]{Jan89}.)

\begin{prop}\label{prop17}
Let $G$ be a compact $p$-adic Lie group and put $R = \Z_p[[G]]$.
Then $R$ satisfies Assumption \ref{ass16}.
\end{prop}

\begin{proof}
Note that $R$ is known to be noetherian \cite[V.2.2.4]{Laz65}.
 Take an open subgroup $G'$ of $G$ such that $G'$ is a pro-$p$ group without $p$-torsion.
Putting $\Lambda = \Z_p[[G']]$, we check the conditions in Proposition \ref{prop18}.
The first two conditions can be easily shown (see also \cite[Lemma 2.3]{Jan89}).
For the third condition, \cite[Theorem 4.1]{Bru66} (or \cite[Corollary (5.2.13)]{NSW08}) shows $\pd_{\Lambda}(M) \leq 1 + \cd_p(G')$ for any $\Lambda$-module $M$, where $\cd_p$ denotes the $p$-cohomological dimension.
Moreover, it is known that, since $G'$ is $p$-torsion free, $\cd_p(G')$ coincides with the dimension of $G'$ as a $p$-adic Lie group, which is finite (\cite[V.2.2.8]{Laz65} and \cite{Ser65}).
Hence we can apply Proposition \ref{prop18} to confirm Assumption \ref{ass16}.
\end{proof}

Recall that we are given an Ore set $S$ and in Subsection \ref{subsec11} we defined the categories $\PP = \PP_{R, S} \subset \MM = \MM_{R, S}$.
Define full subcategories $\CC = \CC_{R, S}$ and $\QQ = \QQ_{R, S}$ of $\MM$ by
\[
\CC = \{ M \in \MM \mid \rpd_R(M) \leq 1 \}
\]
and
\[
\QQ = \{ M \in \MM \mid \pd_R(M) < \infty \}.
\]
When necessary, we write $\CC^{\Left}$ (resp. $\CC^{\Right}$) and so on to stress that the modules are left (resp. right) modules.

Note that Assumption \ref{ass16} implies that $\rpd_R(M) < \infty$ for any $M \in \MM$.
Lemma \ref{lem11} shows that $\PP = \CC \cap \QQ$, meaning that a module is contained in $\PP$ if and only if it is contained both in $\CC$ and in $\QQ$.

Proposition \ref{prop23} will show that $(-)^*$ gives a duality between $\CC^{\Left}$ and $\CC^{\Right}$ under Assumption \ref{ass16}.
Before going to the general case, we give some special cases (Example \ref{eg14}).

\begin{lem}\label{lem12}
Suppose that $R$ contains a subring $\Lambda$ as in Proposition \ref{prop18}.
Then $\CC = \{ M \in \MM \mid \pd_{\Lambda}(M) \leq 1 \}$.
\end{lem}

\begin{proof}
For any $M \in \MM$, by Lemma \ref{lem11} applied to $\Lambda$ instead of $R$, we have $\pd_{\Lambda}(M) = \rpd_{\Lambda}(M)$.
Moreover, Proposition \ref{prop18}(1) shows $\rpd_{\Lambda}(M) = \rpd_{R}(M)$.
These prove the lemma.
\end{proof}

\begin{eg}\label{eg14}
Consider $R = \Z_p[[G]], \Lambda = \Z_p[[G']]$ with $G'$ isomorphic to $\Z_p^d$, and $S = \Lambda \setminus \{0\}$ as in Example \ref{ega1}.
By Proposition \ref{prop17}, Assumption \ref{ass16} holds.
Note that the categories $\MM, \CC, \PP, \QQ$ do not depend on the choice of $G'$.
By Lemma \ref{lem12}, we have $\CC = \{M \in \MM \mid \pd_{\Lambda}(M) \leq 1 \}$.

Suppose $d = 0$, namely $G$ is a finite group and $R = \Z_p[G], \Lambda = \Z_p$.
Then we have $\CC = \MM$, which is the category of all $R$-modules of finite order.
For $M \in \CC$, Proposition \ref{prop18}(1) shows that $M^* \simeq \Ext^1_{\Z_p}(M, \Z_p)$.
Using the exact sequence $0 \to \Z_p \to \Q_p \to \Q_p/\Z_p \to 0$, we have 
\[
M^* \simeq \Hom_{\Z_p}(M, \Q_p/\Z_p),
\]
which is nothing other than the Pontryagin dual.

Suppose $d = 1$. 
It is known that $\pd_{\Lambda}(M) \leq 1$ if and only if $M$ does not contain a nonzero finite $\Lambda$-submodule (\cite[Proposition (5.3.19)(i)]{NSW08}).
In this case, the duality $(-)^*$ on $\CC$ is nothing other than the Iwasawa adjoint.
If $M \in \CC$ is finitely generated and free over $\Z_p$ (namely the $\mu$-invariant of $M$ vanishes), then we have another description
\[
M^* \simeq \Hom_{\Z_p}(M, \Z_p).
\]
\end{eg}

We come back to the general case.

\begin{lem}\label{lem21}
Let $0 \to M' \to M \to M'' \to 0$ be an exact sequence of $R$-modules.

(1) If two of $M, M', M''$ are in $\QQ$, then the other is also in $\QQ$.

(2) Suppose that $M''$ is in $\CC$.
Then $M$ is in $\CC$ if and only if $M'$ is in $\CC$.
\end{lem}

\begin{proof}
These assertions are shown similarly as in Lemma \ref{lem2a}.
\end{proof}

\begin{lem}\label{lem22}
(1) For any $f \in S$, we have an isomorphism $(R/Rf)^* \simeq R/fR$.

(2) For any $M \in \MM^{\Left}$, we have $M^* \in \MM^{\Right}$.

(3) For any $P \in \PP^{\Left}$, we have $P^* \in \PP^{\Right}$.
\end{lem}

\begin{proof}
(1) The exact sequence $0 \to R \overset{\times f}{\to} R \to R/Rf \to 0$ of left $R$-modules and $R^+ \simeq R$ induce an exact sequence $0 \to R \overset{f \times}{\to} R \to (R/Rf)^* \to 0$ of right $R$-modules.
Thus $(R/Rf)^* \simeq R/fR$.

(2) Construct a surjective homomorphism $P = \bigoplus_{i=1}^r R/Rf_i \to M$ as in the proof of Lemma \ref{lem2b}(2).
Since $P$ is $S$-torsion, the kernel $N$ of $P \to M$ is also $S$-torsion and thus $N^+ = 0$.
Hence we obtain an injective map $M^* \to P^*$.
Since $P^* \simeq \bigoplus_{i=1}^r R/f_i R$ by (1), we have $M^* \in \MM^{\Right}$.  

(3) $P^* \in \MM^{\Right}$ follows from (2).
Take an exact sequence $0 \to F_1 \to F_0 \to P \to 0$ of left $R$-modules with $F_1, F_0$ finitely generated and projective.
It induces an exact sequence $0 \to F_0^+ \to F_1^+ \to P^* \to 0$ of right $R$-modules.
Since $F_0^+$ and $F_1^+$ are projective, we conclude that $P^* \in \PP^{\Right}$.
\end{proof}

{\it In the rest of this subsection, we assume Assumption \ref{ass16}}.

\begin{lem}\label{lem26}
Suppose that a finitely generated $R$-module $M$ fits into an exact sequence
\[
0 \to M \to F_0 \to F_1 \to \cdots
\]
with $F_j$ a finitely generated free $R$-module.
Then $\rpd_R(M) \leq 0$.
\end{lem}

\begin{proof}
Let $d(R)$ be the integer as in Assumption \ref{ass16} and let $L$ be the image of $F_{d(R)-1} \to F_{d(R)}$.
Then $E^i(M) \simeq E^{i+d(R)}(L) = 0$ for $i > 0$.
\end{proof}

\begin{prop}\label{prop23}
For any $M \in \CC^{\Left}$, we have $M^* \in \CC^{\Right}$. 
The functors $(-)^*: \CC^{\Left} \to \CC^{\Right}$ and $(-)^*: \CC^{\Right} \to \CC^{\Left}$ are contravariant equivalences of categories such that $(-)^{**} \simeq \id$.
\end{prop}

\begin{proof}
Take an exact sequence 
\[
\dots \to F_2 \to F_1 \to F_0 \to M \to 0
\]
with $F_j$ a finitely generated free $R$-module.
Let $N$ be the image of the map $F_1 \to F_0$.
Then the exact sequence $0 \to N \to F_0 \to M \to 0$ and $M \in \CC$ imply that $\rpd_R(N) \leq 0$.
Therefore the exact sequence $\dots \to F_2 \to F_1 \to N \to 0$ induces an exact sequence 
\[
0 \to N^+ \to F_1^+ \to F_2^+ \to \cdots.
\]
Now Lemma \ref{lem26} implies that $\rpd_R(N^+) \leq 0$.
Since $0 \to F_0^+ \to N^+ \to M^* \to 0$ is exact, we obtain $\rpd_R(M^*) \leq  1$.
Combining with $M^* \in \MM^{\Right}$ from Lemma \ref{lem22}(2), we get $M^* \in \CC^{\Right}$.
Moreover, $0 \to N^{++} \to F_0^{++} \to M^{**} \to 0$ is exact.

Now Lemma \ref{lem26} implies that, if we divide the sequence $0 \to N^+ \to F_1^+ \to F_2^+ \to \cdots$ into short exact sequences, every term satisfies the vanishing of $E^1$.
Therefore we obtain an exact sequence
\[
\dots \to F_2^{++} \to F_1^{++} \to N^{++} \to 0.
\]
Consequently we have an exact sequence
\[
\dots \to F_2^{++} \to F_1^{++} \to F_0^{++} \to M^{**} \to 0.
\]
Finally $F_j^{++} \simeq F_j$ yields $M^{**} \simeq M$.
\end{proof}

\begin{rem}\label{remzz}
(1) In the proof, we essentially showed that, if $N$ is a finitely generated $R$-module with $\rpd_R(N) \leq 0$, then $N^{++} \simeq N$.
In particular, if $\rpd_R(N) = -\infty$, then $N^+=0$ shows $N = 0$, as already mentioned before Lemma \ref{lem11}.

(2) When $R$ is Auslander regular, then Proposition \ref{prop23} is a special case of \cite[Proposition 3.9]{Ven02}.
In this paper, the situation that $R$ is not Auslander regular is important, so we cannot directly apply that result.
However, the author thinks that a similar proof as in \cite{Ven02} is possible.
\end{rem}

Using this duality on $\CC$, we can prove the following lemmas.

\begin{lem}\label{lem24}
Let $0 \to M' \to M \to M'' \to 0$ be an exact sequence in $\CC$.
If two of $M', M, M''$ are in $\PP$, then the other is also in $\PP$.
\end{lem}

\begin{proof}
By Lemma \ref{lem2a}(2), we only have to consider the case where $M', M \in \PP$.
The given sequence induces an exact sequence $0 \to (M'')^* \to M^* \to (M')^* \to 0$.
We have $M^*, (M')^* \in \PP$ by Lemma \ref{lem22}(3), so $(M'')^* \in \PP$ by Lemma \ref{lem2a}(2).
Finally Lemma \ref{lem22}(3) and Proposition \ref{prop23} show $M'' \simeq (M'')^{**} \in \PP$.
\end{proof}

\begin{lem}\label{lem25}
Let $M$ be a module in $\CC$.

(1) There is an exact sequence $0 \to N \to P \to M \to 0$ in $\CC$ with $P \in \PP$.

(2) There is an exact sequence $0 \to M \to P \to N \to 0$ in $\CC$ with $P \in \PP$.
\end{lem}

\begin{proof}
(1) This follows from Lemma \ref{lem2b}(2) and Lemma \ref{lem21}(2).

(2) Construct an exact sequence of (1) with $M^*$ instead of $M$, and take the dual.
\end{proof}

Recall that $Z(R)$ denotes the center of the ring $R$.
We denote by $\Ann_{Z(R)}(M)$ the set of elements of $Z(R)$ which annihilate $M$.
We record here a lemma.

\begin{lem}
For $M \in \CC$, we have $\Ann_{Z(R)}(M^*) = \Ann_{Z(R)}(M)$. 
\end{lem}

\begin{proof}
 By the duality, it is enough to show one inclusion.
 We take any $a \in \Ann_{Z(R)}(M)$ and we will show that $a$ annihilates $M^*$.
 Let $0 \to N \to F \to M \to 0$ be an exact sequence of finitely generated $R$-modules with $F$ free, which induces an exact sequence $0 \to F^+ \to N^+ \to M^* \to 0$.
 We regard $N, F^+$ as submodules of $F, N^+$, respectively.
For any $\phi \in N^+$, let us show that $\phi a \in F^+$.
Choose a basis $e_1, \dots, e_r$ of $F$.
Since $M$ is annihilated by $a$, we have $a e_i \in N$ for $1 \leq i \leq r$.
Let $\phi': F \to R$ be the unique $R$-homomorphism which sends $e_i$ to $\phi(a e_i)$.
Since $a$ is central, $\phi'$ satisfies $\phi'(x) = a\phi(x)$ for any $x \in N$.
In fact, if $x = \sum_{i=1}^r b_i e_i$ for $b_i \in R$, then
\[
\phi'(x) = \phi' \left(\sum_i b_ie_i \right) = \sum_i b_i \phi(ae_i) = \phi \left(\sum_i b_i a e_i \right) = \phi(ax) = a \phi(x).
\]
Then for any $x \in N$, we have
\[
(\phi a)(x) = \phi(x)a = a \phi(x) = \phi'(x). 
\]
This shows $\phi a = \phi' \in F^+$, which completes the proof.
\end{proof}

\subsection{Shift}
Let $R, S$ be as in Subsection \ref{subsec21} satisfying Assumption \ref{ass16}.
We define a kind of axiomatic Fitting invariants which is slightly different from Subsection \ref{subsec11}.

\begin{defn}
A quasi-Fitting invariant is a map $\FF: \CC \to \Omega$ where $\Omega$ is a commutative monoid, satisfying the following conditions:
\begin{itemize}
\item If $P \in \PP$, then $\FF(P) \in \Omega^{\times}$.
\item If $0 \to M' \to M \to P \to 0$ is an exact sequence in $\CC$ with $P \in \PP$, then $\FF(M) = \FF(P)\FF(M')$.
\item If $0 \to P \to M \to M' \to 0$ is an exact sequence in $\CC$ with $P \in \PP$, then $\FF(M) = \FF(P)\FF(M')$.
\end{itemize}
\end{defn}

While a Fitting invariant is defined on $\MM$, a quasi-Fitting invariant is defined only on $\CC$.
On the other hand, the third condition of a quasi-Fitting invariant is additional.
Note also that similar remarks as in Remark \ref{rem2a} are valid.

Here we explain the reasons why we introduce quasi-Fitting invariants.
First of all, the axioms of a quasi-Fitting invariant $\FF$ is exactly what we need to construct $n$-th shift $\SF{n}$ for any $n \in \Z$.
It should be remarked that, if we only consider $\FF = \Fitt_R$ or $\FF = \Fitt_R^{\max}$ as in Subsection \ref{subsec41} or \ref{subsec42}, then the restriction of $\FF$ to $\CC$ is a quasi-Fitting invariant as we will show in Proposition \ref{prop35}.
Hence in that case, essentially we do not have to introduce quasi-Fitting invariants. 
However, we are also interested in Iwasawa theory for $p$-adic Lie extension \cite{CFKSV05}.
In that case, we will construct a quasi-Fitting invariant in Subsection \ref{subsec43}, but the construction is only valid for modules in $\CC$.
Thus we have to deal with quasi-Fitting invariants which are not induced by Fitting invariants.

\begin{prop}\label{prop35}
Suppose that $S$ is central in $R$.
Then the third condition of a quasi-Fitting invariant is implied by the others.
In particular, for a Fitting invariant $\FF: \MM \to \Omega$, the restriction $\FF|_{\CC}: \CC \to \Omega$ is a quasi-Fitting invariant.
\end{prop}

\begin{proof}
Let $0 \to P \to M \to M' \to 0$ be an exact sequence in $\CC$ with $P \in \PP$.
Since $S$ is central, there is an element $f \in S$ which annihilates $M$ and thus $P$.
By Lemma \ref{lem24} and observing the proof of Lemma \ref{lem25}(2), we can construct an exact sequence $0 \to P \to P' \to P'' \to 0$ in $\PP$ with $P' \simeq (R/Rf)^r$ for some $r$.
Now construct a diagram with exact rows and columns
\[
\xymatrix{
0 \ar[r] & P \ar[r] \ar@{^{(}->}[d] & M \ar[r] \ar@{^{(}->}[d] & M' \ar[r] \ar@{=}[d] & 0 \\
0 \ar[r] & P' \ar[r] \ar@{->>}[d] & N \ar[r] \ar@{->>}[d] & M' \ar[r] & 0 \\
& P'' \ar@{=}[r] & P'' & & 
}
\] 
Here $N$ is defined as the push-out of $M$ and $P'$ over $P$, which is in $\CC$ by Lemma \ref{lem21}(2).

Now we show that the middle row splits.
Since $f$ is central, $P'$ is annihilated by $f$ as is $M$, so by the construction, $N$ is also annihilated by $f$.
Therefore the middle row can be seen as a sequence of $R/Rf$-modules, where $R/Rf$ is a ring as $f$ is central.
Then taking the dual, $0 \to (M')^* \to N^* \to (P')^* \to 0$ is an exact sequence of $R/fR$-modules with $(P')^*$ free by Lemma \ref{lem22}(1).
Hence it splits and taking the dual we obtain a splitting of the original sequence $0 \to P' \to N \to M' \to 0$, as claimed.

Therefore, by the second condition of a quasi-Fitting invariant and $P' \in \PP$, we have $\FF(N) = \FF(P') \FF(M')$.
Consequently we can compute
\[
\FF(M) = \FF(P'')^{-1}\FF(N) = \FF(P'')^{-1}\FF(P')\FF(M') = \FF(P)\FF(M'),
\]
as desired.
\end{proof}

Therefore Subsections \ref{subsec41} and \ref{subsec42} give examples of quasi-Fitting invariants.
The author is not sure about the validity of Proposition \ref{prop35} without the hypothesis that $S$ is central in $R$.

\begin{rem}
In the situation of Subsection \ref{subsec41}, the validity of the third condition of a quasi-Fitting invariant is a generalization of \cite[Proposition 2.3]{GK15}, which was a key proposition to prove Theorem \ref{thm01}.
The original proof was rather technical and, in particular, cannot be directly generalized to the noncommutative situations.

\end{rem}

We obtain the shifts of a quasi-Fitting invariant as follows.
Compare it with Theorem \ref{thm41new}.

\begin{thm}\label{thm47}
Let $\FF: \CC \to \Omega$ be a quasi-Fitting invariant.
Then for any $n \in \Z$, the following map $\SF{n}: \CC \to \Omega$ is well-defined and, moreover, $\SF{n}$ is again a quasi-Fitting invariant.

(1) Suppose $n \geq 0$.
For each $M \in \CC$, take an exact sequence
\[
0 \to N \to P_1 \to \dots \to P_n \to M \to 0
\]
in $\CC$ with $P_1, \dots, P_n \in \PP$ and define
 \[
 \SF{n}(M) = \left(\prod_{i=1}^n \FF(P_i)^{(-1)^i}\right) \FF(N).
 \]
 
(2) Suppose $n < 0$.
For each $M \in \CC$, take an exact sequence
\[
0 \to M \to P_{-n} \to \dots \to P_1 \to N \to 0
\]
in $\CC$ with $P_1, \dots, P_{-n} \in \PP$ and define
 \[
 \SF{n}(M) = \left(\prod_{i=1}^{-n} \FF(P_i)^{(-1)^i}\right) \FF(N).
 \]
\end{thm}

\begin{proof}
(1) This is proved similarly as in Theorem \ref{thm41new}.
Namely, we can reduce to the case where $n = 1$, and show that $\SF{1}$ is well-defined and satisfies the first and the second conditions to be a quasi-Fitting invariant.
 To show the third condition, let $0 \to P \to M \to M' \to 0$ be an exact sequence in $\CC$ with $P \in \PP$.
 Take an exact sequence $0 \to N \to P' \to M' \to 0$ in $\CC$ with $P' \in \PP$.
Then we can construct 
\[
\xymatrix{
0 \ar[r] & P \ar[r] & M \ar[r] & M' \ar[r] & 0 \\
0 \ar[r] & P \ar[r] \ar@{=}[u] & P'' \ar@{->>}[u] \ar[r]  & P' \ar[r] \ar@{->>}[u]& 0 \\
& & N \ar@{=}[r] \ar@{^{(}->}[u] & N \ar@{^{(}->}[u] &
}
\]
where $P''$ is the pull-back of $M$ and $P$ over $M'$.
Since $P'' \in \PP$, we obtain 
\[
\SF{1}(M) = \FF(P'')^{-1}\FF(N) = \FF(P)^{-1}\FF(P')^{-1}\FF(N) = \SF{1}(P)\SF{1}(M').
\]

(2) We can prove (2) in a similar way.
Alternatively, using the duality, we can deduce (2) formally from (1) as follows.
If $\FF: \CC^{\Left} \to \Omega$, then a quasi-Fitting invariant $\FF^*: \CC^{\Right} \to \Omega$ can be defined by $\FF^*(M) = \FF(M^*)$.
The axioms are easily checked.
By (1), the quasi-Fitting invariants $(\FF^*)^{\langle -n \rangle}: \CC^{\Right} \to \Omega$ are well-defined for $n < 0$.
Then for $M \in \CC^{\Left}$, it is easy to see that $(\FF^*)^{\langle -n \rangle}(M^*)$ is nothing but $\SF{n}(M)$.
Therefore $\SF{n}(M)$ is well-defined and $\SF{n}$ is a quasi-Fitting invariant.
\end{proof}

The maps $\SF{n}: \CC \to \Omega$ are characterized by the following properties:
\begin{itemize}
\item $\SF{0}$ coincides with $\FF: \CC \to \Omega$.
\item If $0 \to N \to P \to M \to 0$ is an exact sequence in $\CC$ with $P \in \PP$, then we have $\SF{n}(M) = \FF(P)^{(-1)^n} \SF{n-1}(N)$ for any $n \in \Z$.
\item $\SF{n}$ is a quasi-Fitting invariant for any $n \in \Z$.
\end{itemize}

Next we extend the shifts to modules in $\MM$.

\begin{thm}\label{thm41}
Let $\FF: \CC \to \Omega$ be a quasi-Fitting invariant.
Then there is a unique family $\{ \SF{n}: \MM \to \Omega \}_{n \in \Z}$ of maps satisfying the following:
\begin{itemize}
\item $\SF{0}$ is an extension of $\FF: \CC \to \Omega$.
\item If $P \in \QQ$, then $\SF{n}(P) \in \Omega^{\times}$ for any $n \in \Z$.
\item If $0 \to N \to P \to M \to 0$ is an exact sequence in $\MM$ with $P \in \QQ$, then we have $\SF{n}(M) = \SF{n}(P) \SF{n-1}(N)$ for any $n \in \Z$.
\item If $0 \to M' \to M \to P \to 0$ is an exact sequence in $\MM$ with $P \in \QQ$, then $\SF{n}(M) = \SF{n}(P) \SF{n}(M')$ for any $n \in \Z$.
\item If $0 \to P \to M \to M' \to 0$ is an exact sequence in $\MM$ with $P \in \QQ$, then $\SF{n}(M) = \SF{n}(P) \SF{n}(M')$ for any $n \in \Z$.
\end{itemize}
\end{thm}

It is easy to see that $\SF{n}(P) = \SF{0}(P)^{(-1)^n}$ for $P \in \QQ$.
The next corollary immediately follows.

\begin{cor}\label{cor4p}
These maps $\SF{n}$ satisfy the following.
If 
\[
0 \to N \to P_1 \to \dots \to P_d \to M \to 0
\]
 is an exact sequence in $\MM$ with $P_1, \dots, P_d \in \QQ$, then 
 \[
 \SF{n}(M) = \left(\prod_{i=1}^d \SF{n-d+i}(P_i)\right) \SF{n-d}(N)
 \]
  for any $n \in \Z$.
\end{cor}

\begin{rem}
Using the interpretation of $\FF|_{\PP}$ as a group homomorphism $K_0(\PP) \to \Omega^{\times}$ and the resolution theorem $K_0(\PP) \simeq K_0(\QQ)$ in $K$-theory (see \cite[Theorem 3.1.13]{Ros94}), the maps $\SF{0}|_{\PP}$ can be extended to the additive map $\SF{0}|_{\QQ}: K_0(\QQ) \to \Omega^{\times}$ immediately.
This $\SF{0}|_{\QQ}$ can be regarded as a generalization of MacRae's invariant (see \cite[Section 3.6]{Nor76}).
\end{rem}

Now we begin the proof of Theorem \ref{thm41}.
By Theorem \ref{thm47}, we have already defined $\SF{n}(M)$ for $M \in \CC$ and $n \in \Z$.
For $d \geq 1$, let $\CC_d$ (resp. $\PP_d$) be the subcategory of $\MM$ consists of module $M$ (resp. $P$) with $\rpd_R(M) \leq d$ (resp. $\pd_R(P) \leq d$).
Note that $\CC = \CC_1$ (resp. $\PP = \PP_1$) and $\MM = \bigcup_{d} \CC_d$ by Assumption \ref{ass16} (resp. $\QQ = \bigcup_d \PP_d$).
Moreover, $\PP_d = \CC_d \cap \QQ$ by Lemma \ref{lem11}.
Then Theorem \ref{thm41} follows from the following.

\begin{lem}\label{lem44}
Let $d \geq 1$.
Suppose that we are given a family $\{ \SF{n}: \CC_d \to \Omega \}_{n \in \Z}$ of maps satisfying the following.
\begin{itemize}
\item If $P \in \PP_d$, then $\SF{n}(P) \in \Omega^{\times}$ for any $n \in \Z$.
\item If $0 \to N \to P \to M \to 0$ is an exact sequence in $\CC_d$ with $P \in \PP_d$, then we have $\SF{n}(M) = \SF{n}(P) \SF{n-1}(N)$ for any $n \in \Z$.
\item If $0 \to M' \to M \to P \to 0$ is an exact sequence in $\CC_d$ with $P \in \PP_d$, then $\SF{n}(M) = \SF{n}(P) \SF{n}(M')$ for any $n \in \Z$.
\item If $0 \to P \to M \to M' \to 0$ is an exact sequence in $\CC_d$ with $P \in \PP_d$, then $\SF{n}(M) = \SF{n}(P) \SF{n}(M')$ for any $n \in \Z$.
\end{itemize}
Then there is a unique family $\{ \SF{n}: \CC_{d+1} \to \Omega \}_{n \in \Z}$ which is an extension of $\{ \SF{n}: \CC_d \to \Omega \}_{n \in \Z}$ and satisfies the analogous properties replacing $d$ by $d+1$.
\end{lem}

\begin{proof}
For $M \in \CC_{d+1}$, take an exact sequence $0 \to N \to P \to M \to 0$ in $\MM$ with $P \in \PP_d, N \in \CC_d$ and define $\SF{n}(M) = \SF{n}(P) \SF{n-1}(N)$.
In order to show the well-definedness, take another sequence $0 \to N' \to P' \to M \to 0$
and construct a commutative diagram with exact rows and columns
\[
\xymatrix{
0 \ar[r] & N \ar[r] \ar@{=}[d] & P \ar[r] & M \ar[r] & 0 \\
0 \ar[r] & N \ar[r] & L \ar@{->>}[u] \ar[r]  & P' \ar[r] \ar@{->>}[u]& 0 \\
& & N' \ar@{=}[r] \ar@{^{(}->}[u] & N' \ar@{^{(}->}[u]& 
}
\]
where $L$ is defined as the pull-back of $P$ and $P'$ over $M$.
Then $N, P' \in \CC_d$ implies $L \in \CC_d$.
Note that, by the second property, we have $\SF{n-1}(P) = \SF{n}(P)^{-1}$ and $\SF{n-1}(P') = \SF{n}(P')^{-1}$
Thus, by the third property,
\[
\SF{n}(P)\SF{n-1}(N) = \SF{n}(P) \SF{n}(P') \SF{n-1}(L) = \SF{n}(P') \SF{n-1}(N'),
\]
which proves the well-definedness.
This $\SF{n}:\CC_{d+1} \to \Omega$ is clearly an extension of $\SF{n}: \CC_d \to \Omega$.
It is easy to see that if $P \in \PP_{d+1}$ then $\SF{n}(P) \in \Omega^{\times}$.

In order to show the other properties on $\SF{n}$ in Lemma \ref{lem44}, 
let $0 \to M' \to M \to M'' \to 0$ be an exact sequence in $\CC_{d+1}$ such that at least one of $M, M', M''$ is contained in $\PP_{d+1}$.
We claim that there is a commutative diagram with exact rows and columns
\[
\xymatrix{
0 \ar[r] & M' \ar[r] & M \ar[r] & M'' \ar[r] & 0 \\
0 \ar[r] & P' \ar[r] \ar@{->>}[u] & P \ar@{->>}[u] \ar[r]  & P'' \ar[r] \ar@{->>}[u]& 0 \\
0 \ar[r] & N' \ar[r] \ar@{^{(}->}[u] & N \ar@{^{(}->}[u] \ar[r] & N'' \ar[r] \ar@{^{(}->}[u] & 0 
}
\]
where the middle row is in $\PP_d$ and the lower row is in $\CC_d$.
To construct such a diagram, first take surjective homomorphisms $P' \to M'$ and $P'' \to M$ with $P', P'' \in \PP$ using Lemma \ref{lem2b}(2).
Putting $P = P' \oplus P''$, we obtain a trivial exact sequence $0 \to P' \to P \to P'' \to 0$ from which there is a surjective map to the exact sequence $0 \to M' \to M \to M'' \to 0$.
Letting $0 \to N' \to N \to N'' \to 0$ be the kernel of that map, we obtain the desired diagram.

If $M \in \PP_{d+1}$, then $N \in \PP_d$ and we have
\begin{align*}
\SF{n}(M'') & = \SF{n}(P'') \SF{n-1}(N'') \\
&= \SF{n}(P) \SF{n-1}(P') \SF{n-1}(N) \SF{n-2}(N') \\
& = \SF{n}(M) \SF{n-1}(M').
\end{align*}
The other conditions can be confirmed similarly.
\end{proof}

\begin{eg}\label{egzy}
Suppose that $R$ is a regular local ring, so $\CC = \PP$ and $\MM = \QQ$.
Then for the Fitting invariant $\FF = \Fitt_R$, we have $\SF{n}(M) = \ch_R(M)^{(-1)^n}$ for any $M \in \MM$.
Here $\ch_R$ denotes the characteristic ideal.
\end{eg}

\begin{proof}
We may assume that $n = 0$.
Since $M$ admits a $\PP$-resolution of finite length and both $\SF{0}$ and $\ch_R$ are additive with respect to exact sequences, it is enough to show that $\SF{0}(P) = \Fitt_R(P) = \ch_R(P)$ for every $P \in \PP$.
Let $P \to \bigoplus_i R/(f_i)$ be a pseudo-isomorphism of $R$-modules where $f_i$ are nonzero elements of $R$.
For every prime ideal $\qu$ of $R$ of height one, localization at $\qu$ induces an isomorphism $P_{\qu} \simeq \bigoplus_i R_{\qu}/(f_i)$, and thus 
\[
\Fitt_R(P)R_{\qu} = \Fitt_{R_{\qu}}(P_{\qu}) = (\prod_i f_i)R_{\qu} = \ch_R(P)R_{\qu}.
\]
This proves $\Fitt_R(P) = \ch_R(P)$.
\end{proof}

\begin{rem}\label{rem4a}
Suppose that $S$ is central in $R$.
By Proposition \ref{prop35}, a Fitting invariant $\FF$ induces a quasi-Fitting invariant.
Thus we have $\SF{n}: \MM \to \Omega$ for $n \in \Z$ by Theorem \ref{thm41} and also $\varSF{n}: \MM \to \Omega$ for $n \geq 0$ by Theorem \ref{thm41new}.
For $n \geq 0$, the definitions imply $\SF{n}|_{\CC} = \varSF{n}|_{\CC}$.
Moreover, if $M \in \CC_d$, then it can be seen that $\SF{n}(M) = \varSF{n}(M)$ for $n \geq d-1$ using an exact sequence
\[
0 \to N \to P_1 \dots \to P_n \to M \to 0
\]
in $\MM$ with $P_i \in \PP$ and $N \in \CC$.
But we have to mind that in general $\SF{n} \neq \varSF{n}$ for $n \geq 0$.
As a counter-example, when $R = \Z_p[[T]]$, $\FF = \Fitt_R$, and $M = \F_p$, Example \ref{egzy} shows 
\[
\SF{0}(\F_p) = \ch_R(\F_p) = (1) \neq (p, T) = \FF(\F_p) = \varSF{0}(\F_p).
\]
\end{rem}

Let us mention the following property of shifts in a special case.
See Remark \ref{rem4s} for an application.

\begin{lem}\label{lem4r}
 Suppose that $R$ is commutative and contains a subring $\Lambda$ such that
 \begin{itemize}
\item $R = \Lambda[\Delta]$ for a finite cyclic group $\Delta$ generated by $\delta$.
\item The global dimension of $\Lambda$ is finite.
\end{itemize}
Let $\FF: \CC \to \Omega$ be a quasi-Fitting invariant and $n \in \Z$.
Then we have $\SF{n} = \SF{n+2}$ on $\MM$.
\end{lem}

Before the proof of Lemma \ref{lem4r}, we make some observations.
For $R$-modules $M$ and $N$, we define an $R$-module structure on the $\Lambda$-module $M \otimes_{\Lambda} N$ by
 \[
 \rho (m \otimes n) = \rho m \otimes \rho n
 \]
for $\rho \in \Delta, m \in M, n \in N$.
We denote by $M^0$ the $R$-module which is isomorphic to $M$ as $\Lambda$-module but $\Delta$ acts on $M^0$ trivially.
Then we have an isomorphism $R \otimes_{\Lambda} M^0 \simeq R \otimes_{\Lambda} M$ of $R$-modules given by 
\[
\sum_{\rho \in \Delta} \rho \otimes m_{\rho} \leftrightarrow \sum_{\rho \in \Delta} \rho \otimes \rho m_{\rho}.
\]

\begin{proof}[Proof of Lemma \ref{lem4r}]
 The assumptions in Proposition \ref{prop18} are satisfied.
Consider the exact sequence
\[
0 \to \Lambda \to R \overset{\delta - 1}{\to} R \to \Lambda \to 0
\]
of $R$-modules ($\Delta$ acts trivially on $\Lambda$).
For $M \in \MM$, since $\Lambda \otimes_{\Lambda} M \simeq M$ as $R$-modules, we obtain an exact sequence
\[
0 \to M \to R \otimes_{\Lambda} M \to R \otimes_{\Lambda} M \to M \to 0
\]
of $R$-modules.
Since $\pd_{\Lambda}(M) < \infty$, we can take a projective resolution of finite length of $M$ as a $\Lambda$-module.
The resolution can be regarded as a resolution of $M^0$ as an $R$-module.
Since $R \otimes_{\Lambda} \Lambda \simeq R$ and $R \otimes_{\Lambda} M^0 \simeq R \otimes_{\Lambda} M$ as $R$-modules, we obtain a projective resolution of finite length of $R \otimes_{\Lambda} M$ as an $R$-module.
Therefore $R \otimes_{\Lambda} M \in \QQ$.
Consequently the above sequence and Corollary \ref{cor4p} show $\SF{n}(M) = \SF{n+2}(M)$ for $n \in \Z$ as desired.
\end{proof}

\subsection{Modules over a Compact $p$-adic Lie Group}\label{subsec43}
So far we were basically considering the setting of Example \ref{ega1}.
But motivated by the work of \cite{CFKSV05}, it is also interesting to treat more general Iwasawa algebras of compact $p$-adic Lie groups.
In this subsection, we construct a quasi-Fitting invariant over such a noncommutative ring.

Let $G$ be a compact $p$-adic Lie group containing a normal closed subgroup $H$ such that the quotient group $\Gamma = G/H$ is isomorphic to $\Z_p$.
A typical example (as studied in \cite{CFKSV05}) is the case where $G = \GL_2(\Z_p)$ and $H$ is the unique normal closed subgroup containing $\SL_2(\Z_p)$ such that $G/H \simeq \Z_p$.
More generally, for a finite group $\Delta$, we can treat $G = \GL_2(\Z_p) \times \Delta$ and similarly defined $H$.
Such situations naturally occur when we consider the extension of a number field obtained by adjoining the coordinates of $p$-power torsion points of a fixed elliptic curve without complex multiplication, with $H$ corresponding to the cyclotomic $\Z_p$-extension.

We return to the general $G$ with $\Gamma = G/H \simeq \Z_p$.
Put $R = \Z_p[[G]]$, which satisfies Assumption \ref{ass16} by Proposition \ref{prop17}.
In this situation, \cite[Section 2]{CFKSV05} defined the canonical Ore set
\[
\cS = \{ f \in R \mid \text{$R/Rf$ is finitely generated over $\Z_p[[H]]$} \}.
\]
This Ore set satisfies the property that a finitely generated $R$-module is $\cS$-torsion if and only if it is finitely generated over $\Z_p[[H]]$ (\cite[Proposition 2.3]{CFKSV05}).
Also \cite[section 3]{CFKSV05} defined $\cS^* = \bigcup_{n \geq 0} p^n \cS$, which is again an Ore set of $R$.
Then a finitely generated $R$-module $M$ is $\cS^*$-torsion if and only if $M/M(p)$ is $\cS$-torsion where $M(p)$ denotes the submodule of $M$ consists of all elements of finite order.
We use $S = \cS^*$ rather than $\cS$ as the fixed Ore set.

At first we assume that $G$ is one-dimensional, namely $H$ is a finite group.
In this case, $S$-torsion means $\Z_p[[G']]$-torsion for any central open subgroup $G'$ of $G$ isomorphic to $\Z_p$.
Put $\CC_G = \CC_{\Z_p[[G]], S}$ and $\PP_G = \PP_{\Z_p[[G]], S}$.
Then by Propositions \ref{prop4g} and \ref{prop35}, we constructed a quasi-Fitting invariant $\FF_G = \Fitt^{\max}_{\Z_p[[G]]}: \CC_G \to \Omega_G$.

We go back to the general case.
Put $\CC = \CC_{R, S}$ and $\PP = \PP_{R, S}$.
As in the proof of Proposition \ref{prop17}, take any open normal subgroup $G'$ of $G$ such that $G'$ is pro-$p$ and $p$-torsion free.
Put $\Lambda = \Z_p[[G']]$.
Let $\UU$ be the set of open subgroups $U$ of $H \cap G'$ which is normal in $G$.
Note that $\UU$ is a basis of open neighborhoods of the identity element of $H$.

For each $U \in \UU$, since $G/U$ is one-dimensional, we can consider the categories $\CC_{G/U}$ and $\PP_{G/U}$ and the quasi-Fitting invariant $\FF_{G/U}: \CC_{G/U} \to \Omega_{G/U}$.
We will construct a quasi-Fitting invariant $\FF$ for $R$ taking the inverse limit of $\FF_{G/U}$ as follows.

\begin{lem}\label{lem4h}
Let $U \in \UU$.

(1) For each $M \in \CC$, we have $H_1(U, M) = 0$ and $M_U \in \CC_{G/U}$.

(2) For each $P \in \PP$, we have $P_U \in \PP_{G/U}$.
\end{lem}

\begin{proof}
(1) First we show that $M_U \in \MM_{G/U}$, namely $M_U$ is a torsion module over $\Z_p[[G/U]]$.
Since $M$ is $\cS^*$-torsion, the exact sequence $0 \to M(p) \to M \to M/M(p) \to 0$ with $M/M(p)$ finitely generated over $\Z_p[[H]]$ yields an exact sequence
\[
M(p)_U \to M_U \to (M/M(p))_U \to 0
\]
with $(M/M(p))_U$ finitely generated over $\Z_p[H/U]$.
This proves $M_U \in \MM_{G/U}$.

Lemma \ref{lem12} implies that $\pd_{\Lambda}(M) \leq 1$.
Since $\Lambda$ is a local ring, we have a presentation $0 \to F_1 \to F_0 \to M \to 0$ over $\Lambda$ where $F_1, F_0$ are finitely generated and free.
Let $b, a$ be the ranks of $F_1, F_0$, respectively.
We have an exact sequence
\[
0 \to H_1(U, M) \to (F_1)_U \to (F_0)_U \to M_U \to 0,
\]
which shows $b \geq a$.
On the other hand, the dual $M^* \in \CC_G$ has a presentation $0 \to (F_0)^+ \to (F_1)^+ \to M^* \to 0$.
Repeating the above argument, we obtain $a \geq b$ and consequently $a = b$.
Now the above exact sequence shows that $H_1(U, M)$ is a torsion module over $\Z_p[[G'/U]]$.
But $(F_1)_U$ is torsion-free, so we obtain $H_1(U, M) = 0$ and $\pd_{\Lambda_U} (M_U) \leq 1$, from which $M_U \in \CC_{G/U}$ follows using Lemma \ref{lem12}.

(2) Choose an exact sequence $0 \to F_1 \to F_0 \to P \to 0$ over $\Z_p[[G]]$ with $F_1, F_0$ finitely generated and projective.
By (1) we obtain an exact sequence $0 \to (F_1)_U \to (F_0)_U \to P_U \to 0$ with $(F_1)_U, (F_0)_U$ finitely generated and projective over $\Z_p[[G/U]]$.
This proves $P_U \in \PP_{G/U}$.
\end{proof}

If $U, U' \in \UU$ with $U' \subset U$, then we have the natural surjective map $\Z_p[[G/U']] \to \Z_p[[G/U]]$, which induces the natural surjective map $Q(\Z_p[[G/U']]) \to Q(\Z_p[[G/U]])$.
Since both sides are semisimple algebras, this surjective map induces a commutative diagram
\[
\xymatrix{
\GL_n(Q(\Z_p[[G/U']])) \ar@{->>}[r] \ar[d]_{\nr_{U'}} & \GL_n(Q(\Z_p[[G/U]])) \ar[d]_{\nr_{U}} \\
Z(Q(\Z_p[[G/U']]))^{\times} \ar@{->>}[r] & Z(Q(\Z_p[[G/U]]))^{\times}.
}
\]
Therefore we obtain a canonical map $\Omega_{G/U'} \to \Omega_{G/U}$.

\begin{prop}\label{prop3j}
Put $\Omega = \varprojlim_{U \in \UU} \Omega_{G/U}$ and define $\FF: \CC \to \Omega$ by
\[
\FF(M) = \left( \FF_{G/U}(M_U) \right)_U.
\]
This map $\FF$ is well-defined and is a quasi-Fitting invariant.
\end{prop}

\begin{proof}
For $U, U' \in \UU$ with $U' \subset U$, let $\varpi: \Omega_{G/U'} \to \Omega_{G/U}$ be the natural map.
We have to show that $\varpi(\FF_{G/U'}(M_{U'})) = \FF_{G/U}(M_U)$ for $M \in \CC$.

Take a presentation $\Z_p[[G/U']]^b \overset{h}{\to} \Z_p[[G/U']]^a \overset{q}{\to} M_{U'} \to 0$ over $\Z_p[[G/U']]$ such that $\FF_{G/U'}(M_{U'}) = \Fitt_{\Z_p[[G/U']]}(h)$.
It induces an exact sequence $\Z_p[[G/U]]^b \overset{\overline{h}}{\to} \Z_p[[G/U]]^a \overset{\overline{q}}{\to} M_U \to 0$ over $\Z_p[[G/U]]$.
Therefore 
\[
\varpi(\FF_{G/U'}(M_{U'})) = \varpi(\Fitt_{\Z_p[[G/U']]}(h)) = \Fitt_{\Z_p[[G/U]]}(\overline{h}) \subset \FF_{G/U}(M_U).
\]

In order to show the other inclusion, take a presentation $\Z_p[[G/U]]^b \overset{\overline{h}}{\to} \Z_p[[G/U]]^a \overset{\overline{q}}{\to} M_U \to 0$ over $\Z_p[[G/U]]$ such that $\FF_{G/U}(M_U) = \Fitt_{\Z_p[[G/U]]}(\overline{h})$.
We shall show that this sequence can be lifted to $\Z_p[[G/U']]^b \overset{h}{\to} \Z_p[[G/U']]^a \overset{q}{\to} M_{U'} \to 0$ over $\Z_p[[G/U']]$, which would prove
\[
\varpi(\FF_{G/U'}(M_{U'})) \supset \varpi(\Fitt_{\Z_p[[G/U']]}(h)) = \Fitt_{\Z_p[[G/U]]}(\overline{h}) = \FF_{G/U}(M_{U}),
\]
as desired.
In order to construct a lift, first take a lift $q: \Z_p[[G/U']]^a \to M_{U'}$ of $\overline{q}$, whose existence if guaranteed by the freeness of $\Z_p[[G/U']]^a$.
Then $q$ is surjective by Nakayama's lemma.
Putting $\overline{N} = \Ker(\overline{q})$ and $N = \Ker(q)$, we have a commutative diagram with exact rows
\[
\xymatrix{
0 \ar[r] & \overline{N} \ar[r] & \Z_p[[G/U]]^a \ar[r]^-{\overline{q}} & M_U \ar[r] & 0 \\
0 \ar[r] & N \ar[r] \ar[u] & \Z_p[[G/U']]^a \ar[r]^-{q} \ar@{->>}[u] & M_{U'} \ar[r] \ar@{->>}[u] & 0 
}
\]
Since $H_1(U/U', M_{U'}) = 0$ by Lemma \ref{lem4h}(1), the lower sequence yields an exact sequence $0 \to N_{U/U'} \to \Z_p[[G/U]]^a \to M_{U} \to 0$.
Therefore the left vertical map $N \to \overline{N}$ is surjective.
Next, take a lift $\Z_p[[G/U']]^b \to N$ of $\Z_p[[G/U]]^b \to \overline{N}$, which is again surjective by Nakayama's lemma.
Composing with $N \to \Z_p[[G/U']]^a$ gives a lift $h$ of $\overline{h}$, as desired.

We show that $\FF$ is a quasi-Fitting invariant.
If $P \in \PP$, then $P_U \in \PP_{G/U}$ for $U \in \UU$ implies $\FF(P) \in \Omega^{\times}$.
Let $0 \to M' \to M \to P \to 0$ be an exact sequence in $\CC$ with $P \in \PP$.
For any $U \in \UU$, by Lemma \ref{lem4h}(1), it yields an exact sequence $0 \to (M')_U \to M_U \to P_U \to 0$.
Therefore $\FF_{G/U}(M_U) = \FF_{G/U}(P_U) \FF_{G/U}((M')_U)$ since $\FF_{G/U}$ is a quasi-Fitting invariant.
Thus $\FF(M) = \FF(P) \FF(M')$.
The other condition can be confirmed similarly.
\end{proof}

\section{Properties}\label{sec04}

\subsection{Functoriality}\label{subsec44}
In this subsection, we observe a functoriality of shifts.
As an application, in Subsection \ref{subsec57}, we will deduce a result about a finite Galois extension of number fields from a result about an infinite Galois extension, which is a typical method in Iwasawa theory.
As a technical remark, such a reduction theory does not work well on $\MM$, and it is important to treat only modules contained in $\CC$.
In other words, the category $\CC$ is fine not only for the duality theory but also for the reduction theory.

First we show an abstract proposition.
Let $R,S$ and $R',S'$ be as in Subsection \ref{subsec21} satisfying Assumption \ref{ass16}.
Let $\pi: R \to R'$ be a ring homomorphism.
We often denote $\otimes_R$ just by $\otimes$.

\begin{prop}\label{prop51}
Suppose that for any $M \in \CC_{R, S}$, we have $R' \otimes_R M \in \CC_{R', S'}$ and $\Tor_1^R(R', M) = 0$.
Let $\FF_R: \CC_{R, S} \to \Omega_R$ and $\FF_{R'}: \CC_{R', S'} \to \Omega_{R'}$ be quasi-Fitting invariants for $R,S$ and $R', S'$, respectively.
Let $\varpi: \Omega_R \to \Omega_{R'}$ be a monoid homomorphism such that $\varpi(\FF_R(M)) = \FF_{R'}(R' \otimes_R M)$ for $M \in \CC_{R, S}$.
Then we have $\varpi(\SF{n}_R(M)) = \SF{n}_{R'}(R' \otimes_R M)$ for $M \in \CC_{R, S}$ and $n \in \Z$.
\end{prop}

\begin{proof}
The case $n = 0$ is trivial.
First we show that if $P \in \PP_{R, S}$ then $R' \otimes P \in \PP_{R', S'}$.
$R' \otimes P \in \CC_{R', S'}$ is already assumed.
Take an exact sequence $0 \to F_1 \to F_0 \to P \to 0$ of $R$-modules with $F_0, F_1$ finitely generated and projective.
Then $\Tor_1^R(R', P) = 0$ implies that $0 \to R' \otimes F_1 \to R' \otimes F_0 \to R' \otimes P \to 0$ is also exact.
Therefore $\pd_{R'}(R' \otimes P) \leq 1$.

Consider an exact sequence $0 \to N \to P \to M \to 0$ in $\CC_{R, S}$ with $P \in \PP_{R, S}$.
We have
\[
\varpi(\SF{n+1}_R(M)) = \varpi(\FF_R(P)^{(-1)^{n+1}} \SF{n}_R(N)) = \FF_{R'}(R' \otimes P)^{(-1)^{n+1}} \varpi(\SF{n}_R(N)).
\]
On the other hand, since $0 \to R' \otimes N \to R' \otimes P \to R' \otimes M \to 0$ is exact and $R' \otimes P \in \PP_{R', S'}$,
\[
\SF{n+1}_{R'}(R' \otimes M) = \FF_{R'}(R' \otimes P)^{(-1)^{n+1}} \SF{n}_{R'}(R' \otimes N).
\]
Hence the assertion for $n$ is equivalent to that for $n+1$, which proves the proposition.
\end{proof}

We apply this proposition to Iwasawa algebras.
Application to $\FF$ in Proposition \ref{prop3j} is immediate.
Namely, for any $U \in \UU$, letting $\varpi: \Omega \to \Omega_{G/U}$ be the natural map, we obtain $\varpi(\SF{n}(M)) = \SF{n}_{G/U}(M_U)$ for $M \in \CC$ and $n \in \Z$.

As another application, let $R = \Z_p[[G]], \Lambda = \Z_p[[G']]$ be as in Example \ref{ega1} with $d \geq 1$, but $S$ is chosen as follows.
Put $R' = \Z_p[G/G']$ and let $\pi: R \to R'$ be the natural surjective map.
Put $S' = \Z_p \setminus \{0\}$ and let $S$ be the inverse image of $S'$ under the augmentation map $\Lambda \to \Z_p$.

\begin{prop}\label{prop4n}
For any $M \in \CC_{R, S}$, we have $R' \otimes_R M = M_{G'} \in \CC_{R', S'}$ and $\Tor_1^R(R', M) = H_1(G', M) = 0$.
\end{prop}

\begin{proof}
Choose a $\Z_p$-basis $g_1, \dots, g_d$ of $G'$ and put $G'_i = \langle g_1, \dots, g_i \rangle \subset G'$ for $0 \leq i \leq d$.
Moreover put $R_i = \Z_p[[G/G'_i]]$ and let $S_i$ be the inverse image of $S'$ under the augmentation map $\Z_p[[G'/G'_i]] \to \Z_p$.

Since $M \in \CC_{R, S}$, we have an exact sequence $0 \to \Z_p[[G']]^a \to \Z_p[[G']]^a \to M \to 0$ over $\Z_p[[G']]$.
Let us show inductively that the reduction $0 \to \Z_p[[G'/G'_i]]^a \to \Z_p[[G'/G'_i]]^a \to M_{G'_i} \to 0$ is also exact.
The case $i=0$ is trivial.
Suppose $i \geq 1$ and the assertion holds for $i-1$.
By applying the snake lemma to the exact sequence $0 \to \Z_p[[G'/G'_{i-1}]]^a \to \Z_p[[G'/G'_{i-1}]]^a \to M_{G'_{i-1}} \to 0$, we obtain
\[
0 \to (M_{G'_{i-1}})^{\langle g_i \rangle} \to \Z_p[[G'/G'_i]]^a \to \Z_p[[G'/G'_i]]^a \to M_{G'_i} \to 0.
\]
Observe that the module $M_{G'_i}$ is $S_i$-torsion since the image of $S$ under the map $R \to R_i$ is $S_i$.
Therefore we have $(M_{G'_{i-1}})^{\langle g_i \rangle} = 0$, which makes the induction step work.
Now the exactness of  $0 \to \Z_p^a \to \Z_p^a \to M_{G'} \to 0$ shows the assertions.
\end{proof}

\begin{eg}
Suppose that $G$ is commutative.
By Proposition \ref{prop35}, we know that $\FF_R = \Fitt_R: \CC_{R, S} \to \Omega_{R}$ and $\FF_{R'} = \Fitt_{R'}: \CC_{R', S'} \to \Omega_{R'}$ are quasi-Fitting invariants.
The natural surjective map $S^{-1}R \to (S')^{-1}R'$ induces a natural homomorphism $\varpi: \Omega_{R} \to \Omega_{R'}$.
By the fundamental property of Fitting ideals concerning base change, Proposition \ref{prop51} can be applied and we obtain $\varpi(\SF{n}_R(M)) = \SF{n}_{R'}(M_{G'})$ for $M \in \CC_{R, S}$ and $n \in \Z$.
\end{eg}

\begin{eg}\label{eg4o}
To treat the case where $G$ is noncommutative, assume that $d = 1$ and that $G$ contains a finite subgroup $H$ such that $\Gamma = G/H$ is isomorphic to $\Z_p$.
Consider the noncommutative Fitting invariant $\FF_R = \Fitt^{\max}_R: \CC_{R, S} \to \Omega_{R}$ and $\FF_{R'} = \Fitt^{\max}_{R'}: \CC_{R', S'} \to \Omega_{R'}$ as in Proposition \ref{prop4g}.

Here we recall a result of Nickel.
In the proof of \cite[Theorem 6.4]{Nic10}, the commutativity of the diagram
\[
\xymatrix@C=50pt{
M_a(R) \ar[r]^-{\nr_{Q(R)}} \ar@{->>}[d]_{\pi} & \II(R) \ar[d]^{\pi}\\
M_a(R') \ar[r]_-{\nr_{Q(R')}} & \II(R')
}
\] 
is shown for any $a \geq 1$.
Here, the left vertical arrow is the natural reduction map and the right vertical arrow $\pi$, which is a ring homomorphism, is defined by $\pi(x) = \sum_{\chi \in \Irr(G/G')} \aug_{\Gamma} \circ j_{\chi}(x) e_{\chi}$ using the notations in \cite[Theorem 6.4]{Nic10}.
The well-definedness of $\pi: \II(R) \to \II(R')$ also follows by the argument there.

Let $\overline{S}$ (resp. $\overline{S'}$) be the image of $S$ (resp. $S'$) under $\nr_{Q(R)}: R \to \II(R)$ (resp. $\nr_{Q(R')}: R' \to \II(R')$).
Then by the commutativity of the above diagram, we have $\pi(\overline{S}) \subset \overline{S'}$.
Therefore we can define a natural homomorphism $\varpi: \Omega_{R} \to \Omega_{R'}$.

We shall show that $\varpi(\FF_R(M)) = \FF_{R'}(M_{G'})$ for $M \in \CC_{R, S}$.
The inclusion $\varpi(\FF_R(M)) \subset \FF_{R'}(M_{G'})$ is already proved in \cite[Theorem 6.4]{Nic10} (this is true for any $M \in \MM_{R, S}$).
In order to show the other inclusion, take a presentation $(R')^b \overset{\overline{h}}{\to} (R')^a \to M_{G'} \to 0$ such that $\FF'(M_{G'}) = \Fitt_{R'}(\overline{h})$.
Using $H_1(G', M) = 0$ by Proposition \ref{prop4n}, a similar argument as in the proof of Proposition \ref{prop3j} shows that this sequence can be lifted to an exact sequence $R^b \overset{h}{\to} R^a \to M \to 0$.
Then by the above commutative diagram, we have
\[
\varpi(\FF_R(M)) \supset \varpi(\Fitt_{R}(h)) = \Fitt_{R'}(\overline{h}) = \FF_{R'}(M_{G'}),
\]
as desired.

Consequently, Proposition \ref{prop51} can be applied to this situation and we obtain $\varpi(\SF{n}(M)) = \FF'^{\langle n \rangle}(M_{G'})$ for $M \in \CC_{R, S}$ and $n \in \Z$.
\end{eg}

Before closing this subsection, we mention a compatibility between duality and reduction as follows.
It is a generalization of \cite[Lemma 5.18]{GP12}.
Recall that we are studying the case $R = \Z_p[[G]], \Lambda = \Z_p[[G']]$ and $S$ as chosen before Proposition \ref{prop4n}.

\begin{prop}\label{prop4s}
 For $M \in \CC_{R, S}$, we have a canonical isomorphism $(M^*)_{G'} \simeq (M_{G'})^*$, namely, $E^1_R(M)_{G'} \simeq E^1_{R'}(M_{G'})$.
\end{prop}

\begin{proof}
 Take an exact sequence
 \[
 \cdots \to F_2 \to F_1 \to F_0 \to M \to 0
 \]
 with $F_j$ a finitely generated free $R$-module and let $N$ be the image of $F_1 \to F_0$.
On the one hand, as in the proof of Proposition \ref{prop23},we obtain an exact sequence
\[
0 \to N^+ \to F_1^+ \to F_2^+ \to \cdots.
\]
If we divide this sequence into short exact sequences, every term is free as a $\Lambda$-module by Lemma \ref{lem26}.
Therefore this sequence induces an exact sequence
\[
0 \to (N^+)_{G'} \to (F_1^+)_{G'} \to (F_2^+)_{G'} \to \cdots .
\]
On the other hand, each term of the long exact sequence
 \[
 \cdots \to (F_2)_{G'} \to (F_1)_{G'} \to N_{G'} \to 0
 \]
satisfies $\rpd_{R'} \leq 0$.
Hence it induces an exact sequence
 \[
0 \to (N_{G'})^+ \to ((F_1)_{G'})^+ \to ((F_2)_{G'})^+ \to \cdots. 
 \]
Since there is a canonical isomorphism $(F^+)_{G'} \simeq (F_{G'})^+$ for a finitely generated free $R$-module $F$, we conclude $(N^+)_{G'} \simeq (N_{G'})^+$.
Finally, the exact sequences $0 \to ((F_0)^+)_{G'} \to (N^+)_{G'} \to (M^*)_{G'} \to 0$ and $0 \to ((F_0)_{G'})^+ \to (N_{G'})^+ \to (M_{G'})^* \to 0$ given by Proposition \ref{prop4n} imply the assertion.
\end{proof}

\subsection{Dual and Shift}\label{subsec:dual}
In this subsection we observe that, if $\FF$ is a ``typical'' quasi-Fitting invariant, then the formula 
\[
\SF{n-2}(M) = \SF{-n}(M^*)
\]
holds for $M \in \CC$ and $n \in \Z$.
Here we have to be careful about whether the module is a left module or a right module.
In order to state such a formula, $\FF$ must be defined on both $\CC^{\Right}$ and $\CC^{\Left}$ and the target monoid $\Omega$ is the same.

\subsubsection{Setup of Subsection \ref{subsec41}}
Consider the situation of Subsection \ref{subsec41}.
Recall that we have the quasi-Fitting invariant $\FF = \Fitt_R: \CC \to \Omega$.
Theorem \ref{thm47} allows us to define the shifts $\SF{n}: \CC \to \Omega$ for $n \in \Z$ (we also constructed $\SF{n}: \MM \to \Omega$ but we do not consider them in this subsection).

\begin{lem}\label{lem45}
For $P \in \PP$ we have $\FF(P^*) = \FF(P)$.
\end{lem}

\begin{proof}
At first we suppose that $R$ is a local ring.
In that case, we have a presentation $0 \to R^a \overset{h}{\to} R^a \to P \to 0$ over $R$.
Then $0 \to R^a \overset{h^T}{\to} R^a \to P^* \to 0$ is exact, where $h^T$ is the transpose of $h$.
Therefore we have 
\[
\Fitt_{R}(P) = (\det(h)) = (\det(h^T)) = \Fitt_{R} (P^*),
\]
as desired.
Next we consider general $R$.
For any $\qu \in \Spec(R)$, we have $\pd_{R_{\qu}}(P_{\qu}) \leq 1$ and $E^1_R(P) \otimes R_{\qu} \simeq E^1_{R_{\qu}}(P_{\qu})$.
Therefore the local case implies
\[
\Fitt_R(P) R_{\qu} = \Fitt_{R_{\qu}} (P_{\qu}) = \Fitt_{R_{\qu}} (E^1_{R_{\qu}}(P_{\qu})) = \Fitt_{R_{\qu}} (E^1_R(P) \otimes R_{\qu}) = \Fitt_R(E^1_{R}(P)) R_{\qu}.
\]
Thus we have $\Fitt_R(P) = \Fitt_R(E^1_{R}(P))$ as claimed.
\end{proof}

\begin{prop}\label{prop4b}
For $M \in \CC$, we have $\SF{-2}(M) = \FF(M^*)$.
\end{prop}

\begin{rem}\label{rem4h}
By the definition of $\SF{-2}(M)$, Proposition \ref{prop4b} claims that, if $0 \to M \to P_1 \to P_2 \to N \to 0$ is an exact sequence in $\CC$ with $P_1, P_2 \in \PP$, then 
\[
\FF(P_1) \FF(N) = \FF(P_2) \FF(M^*).
\]
This assertion is a generalization of \cite[Lemma 5]{BG03} and \cite[Proposition 7.3]{GP15} (to compare with those results, we need the interpretation of the dual as the Pontryagin dual and the Iwasawa adjoint as in Example \ref{eg14}).
Our proof of Proposition \ref{prop4b} simplifies the original proof because we already know the well-definedness of $\SF{-2}$ and hence we can choose any exact sequence of the above form.
\end{rem}

\begin{proof}[Proof of Proposition \ref{prop4b}]
Take $f \in S$ which annihilates $M$ and an exact sequence $R^b \overset{h}{\to} R^a \to M^* \to 0$.
We can construct an exact sequence
\[
0 \to M \to (R/f)^a \overset{\overline{h}^T}{\to} (R/f)^b \to N \to 0
\]
in $\CC$.
Then $(h^T \mid f 1_b) \in M_{b, a+b}(R)$ is a presentation of $N$ over $R$, where $1_b$ denotes the identity matrix of size $b$.
On the other hand, since $\overline{h}$ is a presentation matrix of $M^*$ over $R/f$, $(h \mid f 1_a) \in M_{a, a+b}(R)$ is a presentation of $M^*$ over $R$.
Therefore
\[
\FF(N) = \Fitt_R(h^T \mid f 1_b) = f^{b-a} \Fitt(h \mid f 1_a) = f^{b-a}\FF(M^*),
\]
which proves $\SF{-2}(M) = \FF(M^*)$.
\end{proof}

\begin{cor}\label{cor47}
We have $\SF{n-2}(M) = \SF{-n}(M^*)$ for $M \in \CC$ and $n \in \Z$. 
\end{cor}

\begin{proof}
The case $n = 0$ is Proposition \ref{prop4b}.
Let $0 \to N \to P \to M \to 0$ be an exact sequence in $\CC$ with $P \in \PP$.
We have an exact sequence $0 \to M^* \to P^* \to N^* \to 0$.
Then
\[
\SF{n-2}(M) = \FF(P)^{(-1)^{n-2}} \SF{n-3}(N) = \FF(P)^{(-1)^{n-2}} \SF{(n-1)-2}(N)
\]
and
\[
\SF{-n}(M^*) = \FF(P^*)^{(-1)^{-n}} \SF{-n+1}(N^*) = \FF(P)^{(-1)^{-n}} \SF{-(n-1)}(N^*),
\]
which prove that the assertion for $n$ is equivalent to that for $n-1$.
This completes the proof.
\end{proof}

\begin{rem}\label{rem4s}
 Consider $R = \Lambda[\Delta]$ as in Lemma \ref{lem4r} and $\FF = \Fitt_R$.
 Then we have 
 \[
 \FF(M) = \SF{-2}(M) = \FF(M^*)
 \]
  for $M \in \CC$ by Lemma \ref{lem4r} and Proposition \ref{prop4b}.
This gives an alternative proof of \cite[Theorem A.8]{GK08}, whose proof is bothering, in the rank one case.
However, our proof of $\FF(M) = \FF(M^*)$ does not depend essentially on the shifting theory.
In short, we constructed an exact sequence $0 \to M \to R \otimes_{\Lambda} M \to R \otimes_{\Lambda} M \to M \to 0$ in Lemma \ref{lem4r} and applied the known formula in Remark \ref{rem4h} to it.
\end{rem}

\subsubsection{Setup of Subsection \ref{subsec42}}
Consider the situation of Subsection \ref{subsec42}, under Assumption \ref{ass35}, and the quasi-Fitting invariant $\FF = \Fitt^{\max}_R: \CC \to \Omega$.

\begin{lem}\label{lem4k}
For $P \in \PP$ we have $\FF(P^*) = \FF(P)$.
\end{lem}

\begin{proof}
This is a generalization of \cite[Propositions 5.3.1 and 6.3.1]{Nic10}.
Since we are assuming that $P$ has a quadratic presentation, a similar argument as the first part of the proof of Lemma \ref{lem45} yields the assertion.
\end{proof}

\begin{prop}\label{prop4a}
For $M \in \CC$, we have $\SF{-2}(M) = \FF(M^*)$.
\end{prop}

\begin{rem}\label{rem49}
This proposition can be written down similarly as in Remark \ref{rem4h}.
It is a generalization of \cite[Propositions 5.3.2 and 6.3.2]{Nic10}.
\end{rem}

\begin{proof}[Proof of Proposition \ref{prop4a}]
Take $f \in S$ which annihilates $M$.
Let $R^b \overset{h}{\to} R^a \to M^* \to 0$ with $b \geq a$ be a presentation such that $\Fitt(h) = \Fitt^{\max}(M^*)$.
Then $(R/f)^b \overset{\overline{h}}{\to} (R/f)^a \to M^* \to 0$ is exact.
If we define $N \in \CC$ as the dual of the kernel of $\overline{h}$, we have an exact sequence
\[
0 \to M \to (R/f)^a \overset{\overline{h}^T}{\to} (R/f)^b \to N \to 0.
\]
Put $h_1 = (h^T \mid 0_{b, b-a}) \in M_b(R)$.
Then $\overline{h_1}$ is a presentation matrix of $N$ over $R/f$ and thus $(h_1 \mid f 1_b)$ is a presentation matrix of $N$ over $R$.
Therefore
\[
\Fitt^{\max}(N) \supset \Fitt(h_1 \mid f 1_b) = \Fitt(h_1^T \mid f 1_b) = \nr(f)^{b-a} \Fitt(h \mid f 1_a) = \nr(f)^{b-a} \Fitt^{\max}(M^*),
\]
where the final equality comes from
\[
\Fitt^{\max}(M^*) = \Fitt(h) \subset \Fitt(h \mid f 1_a) \subset \Fitt^{\max}(M^*).
\]
This proves that $\SF{-2}(M) \supset \FF(M^*)$.

In order to prove the other inclusion, take an exact sequence $0 \to M \to P_1 \to P_2 \to N \to 0$ in $\CC$ with $P_1, P_2 \in \PP$.
Since $0 \to N^* \to P_2^* \to P_1^* \to M^* \to 0$ is exact, applying the above result to $N^*$ yields $\SF{-2}(N^*) \supset \FF(N)$, which is equivalent to 
\[
\FF(M^*) \supset \FF(P_1^*) \FF(P_2^*)^{-1} \FF(N) = \FF(P_1) \FF(P_2)^{-1} \FF(N) = \SF{-2}(M).
\]
This completes the proof.
\end{proof}

\begin{cor}\label{cor4l}
We have $\SF{n-2}(M) = \SF{-n}(M^*)$ for $M \in \CC$ and $n \in \Z$. 
\end{cor}

\begin{proof}
The argument of Corollary \ref{cor47} works.
\end{proof}

\subsubsection{Setup of Subsection \ref{subsec43}}
Consider the situation of Subsection \ref{subsec43} and the quasi-Fitting invariant $\FF: \CC \to \Omega$ constructed there.

\begin{prop}
 We have $\SF{n-2}(M) = \SF{-n}(M^*)$ for $M \in \CC$ and $n \in \Z$.
\end{prop}

\begin{proof}
 For any $U \in \UU$, let $\varpi: \Omega \to \Omega_{G/U}$ be the natural map.
 A similar discussion as in Proposition \ref{prop4s} gives a canonical isomorphism $(M^*)_U \simeq (M_U)^*$, namely, $E^1_R(M)_U \simeq E^1_{\Z_p[[G/U]]}(M_U)$.
 By Lemma \ref{lem4h}(1), the assumptions of Proposition \ref{prop51} hold.
 Then 
 \[
 \varpi(\SF{n-2}(M)) = \SF{n-2}(M_U) = \SF{-n}((M_U)^*) = \SF{-n}((M^*)_U) = \varpi(\SF{-n}(M^*)),
 \]
 where the first and forth equalities follows from Proposition \ref{prop51}, second from Corollary \ref{cor4l}, and third from $(M^*)_U \simeq (M_U)^*$.
 Varying $U$, we get the desired equality.
\end{proof}

\subsection{Computation}\label{subseca1}
Let $\Delta$ be a finite abelian group and put $R = \Z_p[[T_1, \dots, T_d]][\Delta], S = \Z_p[[T_1, \dots, T_d]] \setminus \{0\}$ with $d \geq 0$.
Consider the Fitting invariant $\FF = \Fitt_R$.
As explained in Section \ref{sec01}, it is important to compute $\varSF{2}(\Z_p)$ when $d =1$.
The method of computation is already indicated in Theorem \ref{thm01}.
Let us extend the method to compute $\varSF{n}(\Z_p)$ for general $n \geq 0$ when $d = 1$.
This method is also applicable for general $d \geq 1$: see Examples \ref{eg46} and \ref{eg4w}.

Put $R = \Z_p[[T]][\Delta]$.
Fix a decomposition
\[
\Delta = \prod_{i=1}^s \Delta^{(i)}
\]
where $\Delta^{(i)}$ is cyclic and choose a generator $\delta_i$ of $\Delta^{(i)}$. 
For each $i$, we obtain a complex
\[
C^{(i)}_{\bullet}: \cdots \to \Z_p[\Delta^{(i)}] \overset{\N_{\Delta^{(i)}}}{\to} \Z_p[\Delta^{(i)}] \overset{\delta_i - 1}{\to} \Z_p[\Delta^{(i)}] \to 0
\]
(with the final $0$ at degree $-1$), which is acyclic except for degree $0$ and $H_0(C^{(i)}_{\bullet}) \simeq \Z_p$.
Here $\N$ denotes the norm element in the group ring.
Hence the tensor product $C_{\bullet} = C^{(1)}_{\bullet} \otimes \dots \otimes C^{(s)}_{\bullet}$ is a complex which is acyclic except for degree $0$ and $H_0(C_{\bullet}) \simeq \Z_p$.
Moreover, for each $j \geq 0$, $C_j$ is a free $\Z_p[\Delta]$-module of rank $\frac{s (s+1) \dots (s+j-1)}{j!}$ and the presentation matrix of the boundary map $d_j: C_j \to C_{j-1}$ can be written down.
In this way, for each $n \geq 0$,  we obtain an exact sequence
\[
0 \to N_n \to C_{n-1} \to \dots \to C_0 \to \Z_p \to 0
\]
with $C_j \in \PP$.
Thus we obtain
\[
\varSF{n}(\Z_p) = (T)^{t} \FF(N_n) , \ t = \sum_{j = 0}^{n-1} (-1)^{n+j} \frac{s (s+1) \dots (s+j-1)}{j!}.
\]
Since the boundary map $d_{n+1}$ is a presentation of $N_n$ over $\Z_p[\Delta]$, we obtain a presentation of $N_n$ over $R$, and the Fitting ideal $\FF(N_n)$ of $N_n$ can be computed.
Consequently $\varSF{n}(\Z_p)$ can be computed.

Now we present several examples of explicit computations.

\begin{eg}\label{eg4m}
 Consider the case $s=1$, namely, let $R = \Z_p[[T]][\Delta]$ with $\Delta$ cyclic generated by an element $\delta$.
Then the above computation shows
\begin{eqnarray*}
\varSF{n}(\Z_p) &=& (\delta - 1, T) \quad (n \geq 0, \text{even}), \\
\varSF{n}(\Z_p) &=& \left(\frac{\N_{\Delta}}{T}, 1\right) \quad (n \geq 0, \text{odd}).
\end{eqnarray*}
Moreover, Lemma \ref{lem4r} implies
\begin{eqnarray*}
\SF{n}(\Z_p) &=& (\delta - 1, T) \quad (n \in \Z, \text{even}), \\
\SF{n}(\Z_p) &=& \left( \frac{\N_{\Delta}}{T}, 1 \right) \quad (n \in \Z, \text{odd}).
\end{eqnarray*}
Alternatively, we can use the observation in Example \ref{eg4q}.
\end{eg}

\begin{eg}\label{eg4n}
Consider the case $s = 2$, namely, let $R = \Z_p[[T]][\Delta]$ with $\Delta = \Delta^{(1)} \times \Delta^{(2)}$ and $\Delta^{(i)}$ cyclic.
 Choose a generator $\delta_i$ of $\Delta^{(i)}$ and put $\tau_i = \delta_i -1$ for short.
Then a computation shows
\begin{eqnarray*}
\varSF{0}(\Z_p) &=& (\tau_1, \tau_2, T), \\
\varSF{1}(\Z_p) &=& \left(\frac{\N_{\Delta}}{T}, \N_{\Delta^{(1)}}, \N_{\Delta^{(2)}}, \tau_1, \tau_2, T \right), \\
\varSF{2}(\Z_p) &=& \left( \tau_1^2, \tau_1\tau_2, \tau_2^2, \tau_1 \N_{\Delta^{(2)}}, \tau_2 \N_{\Delta^{(1)}}, (\tau_1, \tau_2, \N_{\Delta^{(1)}}, \N_{\Delta^{(2)}}) T, T^2 \right).
\end{eqnarray*}
Of course, the last formula of $\varSF{2}(\Z_p)$ coincides with the computation of \cite{GK15}.
The description of $\varSF{n}(\Z_p)$ for general $n$ seems difficult.
\end{eg}

\begin{eg}\label{eg4q}
 Let $R = \Z_p[[T]][\Delta]$ with $\Delta$ a finite abelian group.
 For $n \leq -2$, using $\Z_p^{*} \simeq \Z_p$, Corollary \ref{cor47} implies $\SF{n}(\Z_p) = \SF{-2-n}(\Z_p) = \varSF{-2-n}(\Z_p)$.
 Hence if we were able to compute $\varSF{n}(\Z_p)$ for all $n \geq 0$, we would also obtain $\SF{n}(\Z_p)$ for all $n \leq -2$.
 
 To compute $\SF{-1}(\Z_p)$, consider the homomorphism $\nu: \Z_p[\Delta] \to \Z_p[\Delta]$ defined as the multiplication by $\N_{\Delta}$.
 Then the image of $\nu$ is isomorphic to $\Z_p$ and thus we have an exact sequence $0 \to \Z_p \to \Z_p[\Delta] \to \Cok(\nu) \to 0$. 
Since $\Cok(\nu)$ has the presentation $\nu$ over $\Z_p[\Delta]$, it has the presentation matrix $(\N_{\Delta}, T)$ over $R$.
Consequently we obtain
\[
\SF{-1}(\Z_p) = T^{-1}(\N_{\Delta}, T) = \left( \frac{\N_{\Delta}}{T}, 1 \right).
 \]
\end{eg}

\begin{eg}\label{eg46}
Let $R = \Z_p[[T_1, T_2]][\Delta]$ with $\Delta$ cyclic.
Take a generator $\delta$ of $\Delta$.
We shall show that
\begin{eqnarray*}
\varSF{0}(\Z_p) &=& \FF(\Z_p) = (\delta - 1, T_1, T_2), \\
\varSF{n}(\Z_p) &=& (\delta - 1, \N_{\Delta}, T_1, T_2) \quad (n \geq 1), \\
\SF{n}(\Z_p) &=& (\delta - 1, \N_{\Delta}, T_1, T_2) \quad (n \in \Z).
\end{eqnarray*}

The first line is confirmed easily.
Next we compute $\varSF{1}(\Z_p)$ and $\varSF{2}(\Z_p)$.
Taking the tensor product of the complexes
\[
\cdots \to \Z_p[\Delta] \overset{\N_{\Delta}}{\to} \Z_p[\Delta] \overset{\delta - 1}{\to} \Z_p[\Delta] \to 0
\]
and
\[
0 \to \Z_p[[T_1]] \overset{T_1}{\to} \Z_p[[T_1]] \to 0,
\]
we obtain a complex
\[
C_{\bullet}: \cdots \to C_3 \overset{d_3}{\to} C_2 \overset{d_2}{\to} C_1 \overset{d_1}{\to} C_0 \to 0
\]
where $C_0 \simeq \Z_p[[T_1]][\Delta]$ and $C_j \simeq \Z_p[[T_1]][\Delta]^2$ for $j \geq 1$ and  
\[
d_1 = (\delta - 1, T_1), \
d_2 = \begin{pmatrix}
0 & \delta - 1 \\ \N_{\Delta} & -T_1
\end{pmatrix}, \ 
d_3 = \begin{pmatrix}
T_1 & \delta - 1 \\ \N_{\Delta} & 0
\end{pmatrix}.
\]
The complex $C_{\bullet}$ is acyclic outside degree $0$ and $H_0(C_{\bullet}) \simeq \Z_p$.
Therefore, for $n \geq 1$, we have an exact sequence
\[
0 \to N_n \to C_{n-1} \to \dots \to C_0 \to \Z_p \to 0
\]
and thus 
\[
\varSF{n}(\Z_p) = (T_2)^{-1} \FF(N_n).
\]
Since $N_1$ has a presentation $d_2$ over $\Z_p[[T_1]][\Delta]$, $N_1$ has a presentation
\[
\begin{pmatrix}
0 & \delta -1 & T_2 & 0 \\ \N_{\Delta} & -T_1 & 0 & T_2
\end{pmatrix}
\]
over $R$.
Thus
\[
\FF(N_1) = T_2(\delta - 1, \N_{\Delta}, T_1, T_2)
\]
and consequently
\[
\varSF{1}(\Z_p) = (\delta - 1, \N_{\Delta}, T_1, T_2).
\]
Similarly, $N_2$ has a presentation
\[
\begin{pmatrix}
T_1 & \delta -1 & T_2 & 0 \\ \N_{\Delta} & 0 & 0 & T_2
\end{pmatrix}
\]
over $R$ and we obtain $\varSF{2}(\Z_p) = (\delta - 1, \N_{\Delta}, T_1, T_2)$.
Finally, Remark \ref{rem4a} and Lemma \ref{lem4r} imply the other formulas.
\end{eg}

\begin{eg}\label{eg4w}
 Let $R = \Z_p[[T_1, T_2]][\Delta]$ with $\Delta = \Delta^{(1)} \times \Delta^{(2)}$ and $\Delta^{(i)}$ cyclic.
 Choose a generator $\delta_i$ of $\Delta^{(i)}$ and put $\tau_i = \delta_i - 1$.
 Then a long computation as in Example \ref{eg46} gives
 \[
 \SF{2}(\Z_p) = (\tau_1, \tau_2, T_1, T_2)(\tau_1, \tau_2, \N_{\Delta^{(1)}}, \N_{\Delta^{(2)}}, T_1, T_2)^2 + (\N_{\Delta})^2.
 \]
\end{eg}

\section{Applications to Iwasawa Theory}\label{sec5}

\subsection{Tate Sequences}\label{subsec51}
In this subsection, we recall a general construction of an arithmetic exact sequence which is suitable to our algebraic theory.
Let $K$ be a finite extension of $\Q$ and $p$ an odd prime number.
Let $\Sigma$ be a finite set of places of $K$ containing all primes above $p$ and all archimedean primes.
Let $K_{\Sigma}$ be the maximal extension of $K$ unramified outside $\Sigma$ and put $G_{K, \Sigma} = \Gal(K_{\Sigma}/K)$.
Let $\LLL$ be a $p$-adic Lie extension of $K$ contained in $K_{\Sigma}$ and put $G = \Gal(\LLL/K)$.
Let $D \simeq (\Q_p/\Z_p)^r$ be a representation of  $G_{K, \Sigma}$.
We denote by $(-)^{\dual}$ the Pontryagin dual.

\begin{thm}[{\cite[Lemma 4.5]{OV02}}]\label{thm61}
Suppose $H^2(K_{\Sigma}/\LLL, D) = 0$.
Then we have an exact sequence 
\[
0 \to H^1(K_{\Sigma}/\LLL, D)^{\dual} \to F_1 \to F_2 \to H^0(K_{\Sigma}/\LLL, D)^{\dual} \to 0
\] 
of finitely generated $\Z_p[[G]]$-modules with $\pd_{\Z_p[[G]]}(F_1) \leq 1$ and $F_2$ free.
\end{thm}

\begin{rem}
The hypothesis $H^2(K_{\Sigma}/\LLL, D) = 0$ is conjectured to be true as long as $\LLL$ contains the cyclotomic $\Z_p$-extension of $K$ (known as the weak Leopoldt conjecture \cite[p.565]{OV02}, \cite[Conjecture 5.2.1]{Gree10}).
It is actually known to be true if moreover $\Gal(K_{\Sigma}/\LLL)$ acts on $D$ trivially (see \cite[Theorem (10.3.25)]{NSW08}).
\end{rem}

\subsection{Totally Real Fields}\label{subsec52}
In this subsection, we give a noncommutative extension of Theorem \ref{thm01}.
Let $K$ be a totally real number field and $p$ an odd prime number.
Let $\LLL$ be a totally real, Galois extension of $K$ which is a finite extension of the cyclotomic $\Z_p$-extension $K^{\cyc}$ of $K$.
Put $G = \Gal(L/K)$, which is a $p$-adic Lie group of dimension $1$.
Also take a finite extension $K'$ of $K$ contained in $L$ such that $L$ is the cyclotomic $\Z_p$-extension of $K'$ and put $G' = \Gal(L/K')$.
We can apply the algebraic theory to $R = \Z_p[[G]]$ and $S = \Z_p[[G']] \setminus \{0\}$, so put $\CC = \CC_{R, S}$ and $\PP = \PP_{R, S}$.
Let $\Sigma$ be a finite set of places of $K$ containing all primes above $p$, primes ramified in $\LLL/K$, and archimedean primes.
The following proposition is known as mentioned in Section \ref{sec01}, but we give a proof for completeness.

\begin{prop}\label{prop7a}
There is an exact sequence
\[
0 \to X_{\Sigma}(\LLL) \to P_1 \to P_2 \to \Z_p \to 0
\]
in $\CC$ with $P_1, P_2 \in \PP$.
\end{prop}

\begin{proof}
Applying Theorem \ref{thm61} to $D = \Q_p/\Z_p$ equipped with trivial Galois action yields an exact sequence
\[
0 \to X_{\Sigma}(\LLL) \to F_1 \to F_2 \to \Z_p \to 0
\]
over $R$ as stated in Theorem \ref{thm61}.
Note that $\Z_p \in \CC$ and in particular $\Z_p$ is $S$-torsion.
By taking some quotient of $F_1, F_2$ by a free module, we obtain an exact sequence
\[
0 \to X_{\Sigma}(\LLL) \to P_1 \to P_2 \to \Z_p \to 0
\]
of finitely generated $R$-modules such that $\pd_{R}(P_1) \leq 1$, $\pd_{R}(P_2) \leq 1$, and $P_2$ is $S$-torsion.
It is known that $X_{\Sigma}(\LLL)$ is an $S$-torsion module by the weak Leopoldt conjecture (see \cite[Theorem (10.3.25)]{NSW08} or \cite[Theorem 13.31]{Was97}).
Thus $P_1$ is also $S$-torsion and we have $P_1, P_2 \in \PP$.
Now $X_{\Sigma}(\LLL) \in \CC$ also follows.
\end{proof}

Now we review the equivariant main conjecture of Ritter and Weiss \cite{RW02}, reformulated by Nickel \cite[Conjecture 2.1]{Nic13}.
Using the sequence in Proposition \ref{prop7a}, put 
\[
\Phi_{L/K, \Sigma} = [P_1] - [P_2] \in K_0(\PP).
\]
Let $\Theta_{L/K, \Sigma} \in Z(Q(R))^{\times}$ be the equivariant $p$-adic $L$-function, which corresponds to $\Phi_S$ in \cite{Nic13}.
Consider the Fitting invariant $\FF = \Fitt^{\max}_R$ (Proposition \ref{prop4g}) and recall the diagram in Remark \ref{rem4j}.
Then the equivariant main conjecture (without uniqueness assertion) claims that $\FF(\Phi_{L/K, \Sigma})$ is generated by $\Theta_{L/K, \Sigma}$.
This conjecture is proved to be true in \cite{RW11}, under the vanishing of the $\mu$-invariant of $X_{\Sigma}(L)$.

Recall that $\FF = \Fitt^{\max}_R$ is also a quasi-Fitting invariant by Proposition \ref{prop35}.
Then Corollary \ref{cor4p} immediately yields the following, which is the noncommutative variant of Theorem \ref{thm01}.

\begin{thm}\label{thm74}
If the $\mu$-invariant of $X_{\Sigma}(L)$ vanishes, then for the quasi-Fitting invariant $\FF = \Fitt^{\max}_R$ and any $n \in \Z$, we have
\[
\SF{n}(X_{\Sigma}(\LLL)) = (\Theta_{L/K, \Sigma})^{(-1)^n} \SF{n+2}(\Z_p).
\]
In particular,
\[
\FF(X_{\Sigma}(\LLL)) = \Theta_{L/K, \Sigma} \SF{2}(\Z_p).
\]

\end{thm}

\begin{rem}
If $\LLL/K$ is commutative, instead of \cite{RW11}, one may use \cite[Theorem 5.6]{GP15} to obtain Theorem \ref{thm74}.
When $\LLL/K$ is noncommutative, \cite[Theorem 3.3]{Nic13} gives a generalization of \cite[Theorem 5.6]{GP15}, which also gives Theorem \ref{thm74}, but the proof of \cite[Theorem 3.3]{Nic13} relies on \cite{RW11}.
\end{rem}

\subsection{Result at Finite Level}\label{subsec57}
Using the results in Subsection \ref{subsec52} and reduction theory in Subsection \ref{subsec44}, we shall obtain a result for finite Galois extensions of number fields.

Let $K$ be a totally real number field and $p$ an odd prime number.
Let $K'$ be a CM-field which is a finite Galois extension of $K$ and put $\Delta = \Gal(K'/K)$.
Suppose that $K'$ contains $\mu_p$, the $p$-th roots of unity.
Let $\Sigma$ be a finite set of places of $K$ containing all primes above $p$ and primes ramified in $K'/K$.

We have the equivariant zeta function
\[
\theta_{K'/K, \Sigma}(s) = \sum_{\chi \in \Irr(\Delta)} L_{\Sigma}(s, \chi^{-1}) e_{\chi} \in Z(\C[\Delta]),
\]
where $\Irr(\Delta)$ denotes the set of irreducible characters of $\Delta$, $L_{\Sigma}(s, \chi^{-1})$ denotes the $\Sigma$-truncated Artin $L$-function, and $e_{\chi}$ denotes the idempotent.
For an integer $r \geq 2$, a result of Siegel implies that
\[
\theta_{K'/K, \Sigma}(1-r)  \in Z(\Q[\Delta])
\]
(see, for example, \cite[Subsection 1.3]{Nic13}).

Let $L$ be the cyclotomic $\Z_p$-extension of $K'$.
We denote by $j$ the complex conjugation and let $L^+ = L^{\langle j \rangle}$ be the maximal totally real subfield of $L$.
Put $G = \Gal(L/K), G^+ = \Gal(L^+/K) = G / \langle j \rangle, G' = \Gal(L/K') \simeq \Z_p, \Gamma = \Gal(K^{\cyc}/K) \simeq \Z_p$, so $\Delta = G/G'$.
Put $R = \Z_p[[G]], R^+ = \Z_p[[G^+]]$.
Then $R$ is naturally the direct product of $R^+ = \frac{1+j}{2} R$ and $\frac{1-j}{2} R$.

Since $L$ contains $\mu_{p^{\infty}}$, we have the cyclotomic character $\kappa: G \to \Z_p^{\times}$, which is decomposed as $\kappa = \kkappa \omega$ with $\kkappa: G \to 1+p\Z_p$ and $\omega: G \to \mu_{p-1}$, the Teichmuller character.
In general, if $\rho: \G \to \Z_p^{\times}$ is a character of a profinite group $\G$, then we define the continuous automorphism $\rho^{\sharp}$ of $\Z_p$-algebra $\Z_p[[\G]]$ by $\rho^{\sharp}(g) = \rho(g)g$ for $g \in \G$.
For an integer $r$, put $t_r = (\kappa^r)^{\sharp}: R \to R$ and, for a $G$-module $M$, we denote by $M(r)$ the $r$-th Tate twist of $M$.

As a result at infinite level, we have the following.

\begin{prop}\label{prop58}
Let $r$ be an integer.
If the $\mu$-invariant of $X_{\Sigma}(L^+)$ vanishes, then for the quasi-Fitting invariant $\FF_R = \Fitt_R^{\max}$ and any $n \in \Z$, we have
\[
\SF{n}_{R}(X_{\Sigma}(L^+)(-r)) = t_r(\Theta_{L^+/K, \Sigma})^{(-1)^n} \SF{n+2}_{R}(\Z_p(-r)).
\]
\end{prop}

\begin{proof}
We have $\SF{n}_{R}(X_{\Sigma}(L^+)) = (\Theta_{L^+/K, \Sigma})^{(-1)^n} \SF{n+2}_{R}(\Z_p)$ by Theorem \ref{thm74}.
We apply Proposition \ref{prop51} to the ring homomorphism $\pi = t_{-r}: R \to R$, the monoid homomorphism $\varpi: \Omega_R \to \Omega_R$ induced by $t_r$, and quasi-Fitting invariant $\FF_R$.
As in \cite[Lemma 3.1 i]{Pop09}, the property $t_r(\FF_R(M)) = \FF_R(M(-r))$ can be checked.
This completes the proof.
\end{proof}

Next we want to deduce a result for the finite extension $K'/K$.
Put $R' = \Z_p[\Delta]$.
Recall that we have the natural map $\pi: \II(R) \to \II(R')$ in Example \ref{eg4o} defined by $\pi(x) = \sum_{\chi \in \Irr(\Delta)} \aug_{\Gamma} \circ j_{\chi}(x) e_{\chi}$.
We need the following.

\begin{prop}\label{propzw}
 For an integer $r \geq 2$, we have $\pi \circ t_r (\Theta_{L^+/K, \Sigma}) = \theta_{K'/K, \Sigma}(1-r)$.
\end{prop}

\begin{proof}
 This is implicitly used in \cite[p.1245]{Nic13}.
For any character $\chi$ of $\Delta$, what we need is the equality \[
\aug_{\Gamma} \circ j_{\chi} \circ t_r(\Theta_{L^+/K, \Sigma}) = L_{\Sigma}(1-r, \chi).
\]
If the parities of $\chi$ and $r$ are different, then the both sides vanish and we have nothing to do.
Assume that the parities of $\chi$ and $r$ the same.
 By definition, $\Theta_{L^+/K, \Sigma}$ can be characterized by 
 \[
 \aug_{\Gamma} \circ (\kkappa^r)^{\sharp} \circ j_{\phi} (\Theta_{L^+/K, \Sigma}) = L_{p, \Sigma}(1-r, \phi) = L_{\Sigma}(1-r, \phi \omega^{-r})
 \]
 for an integer $r \geq 2$ and an even character $\phi$ of $\Delta$, where $L_{p, \Sigma}$ denotes the $p$-adic $L$-function.
Then letting $\phi = \chi \omega^r$ shows the assertion since 
\[
j_{\chi} \circ t_r = j_{\chi} \circ (\kkappa^r)^{\sharp} \circ (\omega^r)^{\sharp} = (\kkappa^r)^{\sharp} \circ j_{\chi \omega^r}.
\] 
\end{proof}

We obtain the following by applying Example \ref{eg4o} to Proposition \ref{prop58}, using Proposition \ref{propzw}.

\begin{thm}\label{thm5a}
Let $r \geq 2$ be an integer.
If the $\mu$-invariant of $X_{\Sigma}(L^+)$ vanishes, then for the quasi-Fitting invariant $\FF_{R'} = \Fitt^{\max}_{R'}$ and any $n \in \Z$, we have
\[
\SF{n}_{R'}(X_{\Sigma}(L^+)(-r)_{G'}) = \theta_{K'/K, \Sigma}(1-r)^{(-1)^n} \SF{n+2}_{R'}(\Z_p(-r)_{G'}).
\]
\end{thm}

\begin{rem}\label{rem:59a}
Put $e_r = \frac{1 + (-1)^r j}{2}$.
Note that in Proposition \ref{prop58} and Theorem \ref{thm5a}, we only see at $e_r$-parts.
Let $\OO_{K', \Sigma}$ be the ring of $\Sigma$-integers in $K'$ and $H^i_{et}(\OO_{K', \Sigma}, \Z_p(r))$ denote the etale cohomology groups.
Then as in \cite[Proposition 6.18]{GP15}, there are isomorphisms
 \[
 (X_{\Sigma}(L^+)(-r)_{G'})^{\dual} \simeq e_r H^2_{et}(\OO_{K', \Sigma}, \Z_p(r)),
 \]
 where $(-)^{\dual}$ is the Pontryagin dual, and
 \[
 \Z_p(r)_{G'} \simeq e_r H^1_{et}(\OO_{K', \Sigma}, \Z_p(r))_{\tor}.
 \]
Therefore the assertion of Theorem \ref{thm5a} can be restated as
\[
e_r\SF{n}_{R'}(H^2_{et}(\OO_{K', \Sigma}, \Z_p(r))^{\dual}) = \theta_{K'/K, \Sigma}(1-r)^{(-1)^n} \SF{n+2}_{R'}((H^1_{et}(\OO_{K', \Sigma}, \Z_p(r))_{\tor})^{\dual}).
\]
In particular, by putting $n = -2$ and using Proposition \ref{prop4a}, we obtain
\[
e_r\FF_{R'}(H^2_{et}(\OO_{K', \Sigma}, \Z_p(r))) = \theta_{K'/K, \Sigma}(1-r) \FF_{R'}((H^1_{et}(\OO_{K', \Sigma}, \Z_p(r))_{\tor})^{\dual}).
\]
This is nothing but \cite[Corollary 5.13]{Nic13}, which is a noncommutative version of \cite[Theorem 6.11]{GP15}, and they are refinements of the cohomological Coates-Sinnott conjecture.
In those articles, they used the concept of abstract $1$-motives, but our argument in this paper does not use it.
\end{rem}

\subsection{$\Sigma_p$-Ramified Iwasawa Modules}\label{subsec53}
Consider the situation in Subsection \ref{subsec52}, but we assume that $G$ is abelian.
Take a subset $\Sigma'$ of $\Sigma$ containing all primes above $p$, but not necessarily containing all primes which are ramified in $L/K$.
For example $\Sigma' = \Sigma_p$, the set of primes above $p$.
In this subsection, we consider $X_{\Sigma'}(\LLL)$ as an $R$-module.

For $v \in \Sigma \setminus \Sigma'$, put
\[
\UU_v(L) = \varprojlim_F (( \OO_F / v)^{\times} \otimes_{\Z} \Z_p),
\]
where $F$ runs through all intermediate number fields of $L/K$.
The structure of $\UU_v(L)$ is described as follows.
Fix a place $w$ of $L$ above $v$.
Let $F_{v}, E_{v}$ be the decomposition field, inertia field of $w$ in $L/K$, respectively, which are independent of the choice of $w$.
Put $G_v = \Gal(L/F_v)$ and $I_v = \Gal(L/E_v)$, which are the decomposition group and the inertia group, respectively.
Then $\Gal(E_v/F_v)$ is pro-cyclic generated by the (arithmetic) Frobenius element $\sigma_v$.
On the other hand, $I_v$ is a finite group contained in $\Gal(L/K^{\cyc})$.
Let $\NN(v), \NN(w)$ be the orders of the residue fields.
Then the following is easily shown.

\begin{lem}\label{lem5b}
If $p \nmid (\NN(w)-1)$, then $\UU_v(L) = 0$.
 If $p | (\NN(w) - 1)$, then 
 \[
 \UU_v(L) \simeq \Z_p[[G]] \otimes_{\Z_p[[G_v]]} \Z_p(\kappa_v),
 \]
 where $\kappa_v$ is the character of $G_v$ defined as $G_v \to \Gal(E_v/F_v) \to \Z_p^{\times}$ where the last map sends $\sigma_v$ to $\NN(v)$.
\end{lem}

Note that the module written down on the right is nothing but the induction of $\Z_p(\kappa_v)$ from $G_v$ to $G$, often denoted by $\Ind_{G_v}^{G} (\Z_p(\kappa_v))$.
In particular, $\UU_v(L)$ is free of finite rank $[F_v: K]$ as a module over $\Z_p$.
By class field theory and the weak Leopoldt conjecture, we have the following.

\begin{prop}\label{prop75}
There is an exact sequence
\[
0 \to \bigoplus_{v \in \Sigma \setminus \Sigma'} \UU_v(\LLL) \to X_{\Sigma}(\LLL) \to X_{\Sigma'}(\LLL)  \to 0
\]
in $\CC$.
\end{prop}

There is a possibility that none of the three terms is contained in $\PP$, in which case we cannot apply the algebraic theory.
To avoid such an obstruction, we consider only nontrivial character parts.

Put $\Delta = \Gal(L/K^{\cyc})$, which is the torsion part of $G$.
Note that, since $G$ is commutative, $G$ can be decomposed into the product of $\Delta$ and a group isomorphic to $\Z_p$.
Let $\Delta_p$ (resp. $\Delta_{p'}$) be the subgroups of $\Delta$ consisting of elements of $p$-power order (resp. order prime to $p$).
Put $\overline{G} = G / \Delta_{p'}$, which is the pro-$p$ quotient of $G$.

Let $\psi$ be a nontrivial character of $\Delta_{p'}$ and put $\OO_{\psi} = \Z_p[\Image{\psi}]$.
Since the order of $\psi$ is prime to $p$, taking $\psi$-parts denoted by $(-)_{\psi}$ is an exact functor.
Thus we have an exact sequence
\[
0 \to \bigoplus_{v \in \Sigma \setminus \Sigma'} \UU_v(\LLL)_{\psi} \to X_{\Sigma}(\LLL)_{\psi} \to X_{\Sigma'}(\LLL)_{\psi}  \to 0
\]
over $\OO_{\psi}[[\overline{G}]]$, and hence over $R = \Z_p[[G]]$.
Moreover, by Proposition \ref{prop7a}, $X_{\Sigma}(\LLL)_{\psi} \in \PP$ and $[X_{\Sigma}(\LLL)_{\psi}] = (\Phi_{L/K, \Sigma})_{\psi}$ in $K_0(\PP)$, where $(\Phi_{L/K, \Sigma})_{\psi}$ has the obvious meaning.
Consequently, we obtain the following.

\begin{thm}\label{thm76}
For any quasi-Fitting invariant $\FF$, we have
\[
\FF(X_{\Sigma'}(\LLL)_{\psi}) = \FF((\Phi_{L/K, \Sigma})_{\psi}) \SF{-1} \left( \bigoplus_{v \in \Sigma \setminus \Sigma'} \UU_v(\LLL)_{\psi} \right).
\]
\end{thm}

In particular, consider the case $\FF = \Fitt_R$ and we compute the right hand side.
As seen in Subsection \ref{subsec52}, the equivariant main conjecture claims $\FF(\Phi_{L/K, \Sigma}) = (\Theta_{L/K, \Sigma})$.

\begin{lem}\label{lemzv}
Let $\FF = \Fitt_R$.
Take a lift $\widetilde{\sigma_v}$ of $\sigma_v$ to $\Gal(L/F_v)$.
Then we have
\[
\SF{-1} \left( \bigoplus_{v \in \Sigma \setminus \Sigma'}\UU_v(L) \right) = \prod_{v \in \Sigma \setminus \Sigma'}\left( \frac{\N_{I_v}}{\widetilde{\sigma_v}-\NN(v)}, 1 \right) R.
\]
\end{lem}

\begin{proof}
Observe that $\FF$ satisfies $\FF(M_1 \oplus M_2) = \FF(M_1)\FF(M_2)$ (this formula for the noncommutative case is unknown; see \cite[Remark 5.11]{Nic17}).
Then we clearly have
\[
\SF{-1} \left( \bigoplus_{v \in \Sigma \setminus \Sigma'} \UU_v(\LLL) \right) = \prod_{v \in \Sigma \setminus \Sigma'} \SF{-1} ( \UU_v(\LLL)).
\]
By Lemma \ref{lem5b}, 
\[
\SF{-1}(\UU_v(L)) = \SF{-1}_{\Z_p[[G_v]]}(\Z_p(\kappa_v)) R = (\kappa_v^{-1})^{\sharp} \left( \SF{-1}_{\Z_p[[G_v]]} (\Z_p)\right) R.
\]
Here we used Proposition \ref{prop51} twice, applied to the inclusion map $\Z_p[[G_v]] \to R$ and to the automorphism $(\kappa_v^{-1})^{\sharp}: \Z_p[[G_v]] \to \Z_p[[G_v]]$.

$\SF{-1}_{\Z_p[[G_v]]} (\Z_p)$ is computed in Example \ref{eg4q}.
Put $m_v = [E_v: F_v^{\cyc}]$, which is prime to $p$.
Then we have a decomposition $\Gal(L/F_v) = \Gal(L/F_v^{\cyc}) \times \langle \widetilde{\sigma_v}^{m_v} \rangle$.
By Example \ref{eg4q}, we have 
\[
\SF{-1}_{\Z_p[[G_v]]}(\Z_p) = \left( \frac{\N_{\Gal(L/F_v^{\cyc})}}{\widetilde{\sigma_v}^{m_v} - 1}, 1 \right).
\]
Note that each element of $\Gal(E_v/F_v^{\cyc})$ is congruent to $\sigma_v^j$ modulo $\sigma_v^{m_v} - 1$ for a unique $0 \leq j < {m_v}$.
Therefore
\[
\left( \frac{\N_{\Gal(L/F_v^{\cyc})}}{\widetilde{\sigma_v}^{m_v} - 1}, 1 \right) 
= \left( \frac{\N_{I_v}(1+\widetilde{\sigma_v} + \dots + \widetilde{\sigma_v}^{{m_v}-1})}{\widetilde{\sigma_v}^{m_v}-1}, 1 \right)
= \left( \frac{\N_{I_v}}{\widetilde{\sigma_v}-1}, 1 \right).
\]
Since $\kappa_v^{-1}(\widetilde{\sigma_v}) = \NN(v)^{-1}$, the assertion follows.
\end{proof}

\subsection{Imaginary Quadratic Fields}\label{subsec54}
We consider a two-variable analogue of Theorem \ref{thm01}.
Let $K$ be an imaginary quadratic field and $p$ be an odd prime number which splits into two primes $\pe, \overline{\pe}$ in $K$.
Let $\LLL$ be an abelian extension of $K$ which is a finite extension of $\widetilde{K}$, the $\Z_p^2$-extension of $K$.
Put $G = \Gal(\LLL/K)$.
Let $\Sigma$ be a finite set of places of $K$ containing all primes above $p$ and primes ramified in $\LLL/K$.
It is known that the $\Sigma$-ramified Iwasawa module $X_{\Sigma}(\LLL)$ is not a torsion module (\cite[Theorem (10.3.25)]{NSW08}).
Instead, put $\Sigma_0 = \Sigma \setminus \{\overline{\pe}\}$ and we shall study the $\Sigma_0$-ramified Iwasawa module $X_{\Sigma_0}(\LLL)$, which is known to be torsion.
Let $X(L)$ be the unramified Iwasawa module of $L$.

Take an open subgroup $G'$ of $G$ which is isomorphic to $\Z_p^2$.
We can apply the algebraic theory to $R = \Z_p[[G]]$, $\Lambda = \Z_p[[G']]$, and $S = \Lambda \setminus \{0\}$.
Put $\MM = \MM_{R, S}, \CC = \CC_{R, S}$, and $\PP = \PP_{R, S}$.

For a finite place $v$ of $K$, we define an $R$-module $\UU_v(L)$ as follows.
If $v \nmid p$, then it is defined as in Subsection \ref{subsec53}.
If $v | p$, then it is defined by
\[
\UU_{v}(L) = \varprojlim_F \UU_{v}(F)
\]
where $F$ runs over intermediate number fields of $L/K$ and $\UU_v(F)$ denotes the principal semi-local unit group of $F$ above $v$.
Put
\[
\EE(L) = \varprojlim_F (\OO_F^{\times} \otimes_{\Z} \Z_p).
\]

Consider the commutative diagram with exact rows
\[
\xymatrix{
0 \ar[r] & \EE(\LLL) \ar[r] \ar@{=}[d] & \UU_{\pe}(\LLL) \times \UU_{\overline{\pe}}(\LLL) \times \prod_{v \in \Sigma \setminus \Sigma_p} \UU_{v}(\LLL) \ar[r] \ar@{->>}[d] & X_{\Sigma}(\LLL) \ar[r]  \ar@{->>}[d] & X(\LLL) \ar[r] \ar@{=}[d] & 0 \\
0 \ar[r] & \EE(\LLL) \ar[r] & \UU_{\pe}(\LLL) \times \prod_{v \in \Sigma \setminus \Sigma_p} \UU_{v}(\LLL) \ar[r] & X_{\Sigma_{0}}(\LLL) \ar[r] & X(\LLL) \ar[r] & 0
}
\]
where the injectivity follows from the $\pe$-adic Leopoldt conjecture (\cite{Bru67} and \cite[Theorem 10.3.16]{NSW08}).
Then we obtain an exact sequence
\[
0 \to \UU_{\overline{\pe}}(\LLL) \to X_{\Sigma}(\LLL) \to X_{\Sigma_0}(\LLL) \to 0.
\]

\begin{lem}\label{lem72}
We have $\pd_R(\UU_{\overline{\pe}}(\LLL)) < \infty$.
\end{lem}

\begin{proof}
Let $G_{\overline{\pe}}$ be the decomposition group of $\overline{\pe}$ in $G$.
It is crucial that $G_{\overline{\pe}}$ is isomorphic to the direct product of $\Z_p^2$ and a finite abelian group of order prime to $p$.
To show that, first observe that $G_{\overline{\pe}}$ can be regarded as the Galois group of an two-dimensional abelian $p$-adic Lie extension of $K_{\overline{\pe}} \simeq \Q_p$.
As $p$ is odd, the local class field theory shows that the Galois group of the maximal abelian pro-$p$-extension of $\Q_p$ is isomorphic to $\Z_p^2$.
Therefore the pro-$p$ quotient of $G_{\overline{\pe}}$ must be isomorphic to $\Z_p^2$.
This implies the claim.

Fix a prime $\overline{\Pe}$ of $L$ above $\overline{\pe}$.
Then we have $\UU_{\overline{\pe}}(L) \simeq \Z_p[[G]] \otimes_{\Z_p[[G_{\overline{\pe}}]]} \UU_{\overline{\Pe}}(L)$, where the definition of $\UU_{\overline{\Pe}}(L)$ is clear.
By the structure of $G_{\overline{\pe}}$ described above, the global dimension of $\Z_p[[G_{\overline{\pe}}]]$ is finite.
Therefore we have $\pd_R(\UU_{\overline{\pe}}(L)) < \infty$ (at most two actually), as desired.
\end{proof}

\begin{prop}\label{prop71}
There is an exact sequence
\[
0 \to X_{\Sigma_0}(\LLL) \to P_1 \to P_2 \to \Z_p \to 0
\]
in $\MM$ with $X_{\Sigma_0}(\LLL) \in \CC$ and $P_1, P_2 \in \PP$ (Note that $\Z_p$ is not contained in $\CC$ because $\pd_{\Lambda}(\Z_p) = 2$).
\end{prop}

\begin{proof}
First we show $X_{\Sigma_0}(\LLL) \in \CC$.
It is already known that $X_{\Sigma_0}(\LLL) \in \MM$, so it is enough to show that $\pd_{\Lambda}(X_{\Sigma_0}(\LLL)) \leq 1$.
This is essentially proved in \cite[Theorem 23]{Per84} and the preceding remarks there.
Alternatively, we can apply a result by Greenberg \cite[Proposition 4.1.1]{Gree16}.

Applying Theorem \ref{thm61} to $D = \Q_p/\Z_p$ yields
\[
0 \to X_{\Sigma}(\LLL) \to F_1 \to F_2 \to \Z_p \to 0
\]
of finitely generated $R$-modules, where $\pd_{R}(F_1) \leq 1$ and $F_2$ free.
Since $\Z_p$ is a torsion module, by taking the quotient of $F_1, F_2$ by a free module, we can construct an exact sequence
\[
0 \to X_{\Sigma}(\LLL) \to F_1' \to P_2 \to \Z_p \to 0
\]
of finitely generated $R$-modules, where $\pd_{R}(F_1') \leq 1$ and $P_2 \in \PP$.

Let $P_1$ be the cokernel of the composition map $\UU_{\overline{\pe}}(\LLL) \hookrightarrow X_{\Sigma}(\LLL) \hookrightarrow F_1'$, which satisfies $\pd_{R}(P_1) < \infty$ by Lemma \ref{lem72}.
Then we obtain an exact sequence
\[
0 \to X_{\Sigma_0}(\LLL) \to P_1 \to P_2 \to \Z_p \to 0
\]
in $\MM$.
Since $\pd_{\Lambda}(\Z_p) = 2$ and $X_{\Sigma_0}(L) \in \CC$, we have $\pd_{\Lambda}(P_1) \leq 1$.
Thus $P_1 \in \CC \cap \QQ = \PP$, which completes the proof. 
\end{proof}

Using the sequence in Proposition \ref{prop71}, put $\Phi_{L/K, \Sigma_0} = [P_1] - [P_2] \in K_0(\PP)$.

\begin{thm}\label{thmzx}
For any quasi-Fitting invariant $\FF$ and $n \in \Z$, we have
\[
\SF{n}(X_{\Sigma_0}(\LLL)) = \FF(\Phi_{L/K, \Sigma_0})^{(-1)^n} \SF{n+2}(\Z_p).
\]
In particular,
\[
\FF(X_{\Sigma_0}(\LLL)) = \FF(\Phi_{L/K, \Sigma_0}) \SF{2}(\Z_p).
\]
\end{thm}

Moreover, the argument in Subsection \ref{subsec53} can be applied.
Take a subset $\Sigma_0'$ of $\Sigma_0$ containing $\{\pe\}$.
As in Proposition \ref{prop75}, we have an exact sequence
\[
0 \to \bigoplus_{v \in \Sigma_0 \setminus \Sigma_0'} \UU_v(\LLL) \to X_{\Sigma_0}(\LLL) \to X_{\Sigma_0'}(\LLL)  \to 0
\]
in $\CC$ and the structure of $\UU_v(L)$ is described as in Lemma \ref{lem5b}.

Decompose $\Delta = \Gal(L/\widetilde{K}) = \Delta_p \times \Delta_{p'}$ and let $\psi$ be any nontrivial character of $\Delta_{p'}$.
Then as in Theorem \ref{thm76}, for any quasi-Fitting invariant $\FF$ for $R$, we have
\[
\FF(X_{\Sigma_0'}(\LLL)_{\psi}) = \FF((\Phi_{L/K, \Sigma_0})_{\psi}) \SF{-1} \left( \bigoplus_{v \in \Sigma_0 \setminus \Sigma_0'} \UU_v(\LLL)_{\psi} \right).
\]
If $\FF = \Fitt_R$, then as in Lemma \ref{lemzv}
\[
\SF{-1} \left( \bigoplus_{v \in \Sigma_0 \setminus \Sigma_0'} \UU_v(\LLL) \right) = \prod_{v \in \Sigma_0 \setminus \Sigma_0'} \left( \frac{\N_{I_v}}{\widetilde{\sigma_v} - \NN(v)}, 1 \right) R,
\]
where the notations are similar to those in Subsection \ref{subsec53}.

{\small
\bibliographystyle{abbrv}
\bibliography{fitting}

\begin{thebibliography}{10}

\bibitem{Bru66}
A.~Brumer.
\newblock Pseudocompact algebras, profinite groups and class formations.
\newblock {\em J. Algebra}, 4:442--470, 1966.

\bibitem{Bru67}
A.~Brumer.
\newblock On the units of algebraic number fields.
\newblock {\em Mathematika}, 14:121--124, 1967.

\bibitem{BG03}
D.~Burns and C.~Greither.
\newblock Equivariant {W}eierstrass preparation and values of {$L$}-functions
  at negative integers.
\newblock {\em Doc. Math.}, (Extra Vol.):157--185, 2003.
\newblock Kazuya Kato's fiftieth birthday.

\bibitem{CFKSV05}
J.~Coates, T.~Fukaya, K.~Kato, R.~Sujatha, and O.~Venjakob.
\newblock The {$\rm GL_2$} main conjecture for elliptic curves without complex
  multiplication.
\newblock {\em Publ. Math. Inst. Hautes \'Etudes Sci.}, (101):163--208, 2005.

\bibitem{CG98}
P.~Cornacchia and C.~Greither.
\newblock Fitting ideals of class groups of real fields with prime power
  conductor.
\newblock {\em J. Number Theory}, 73(2):459--471, 1998.

\bibitem{CR90}
C.~W. Curtis and I.~Reiner.
\newblock {\em Methods of representation theory. {V}ol. {I}}.
\newblock Wiley Classics Library. John Wiley \& Sons, Inc., New York, 1990.
\newblock With applications to finite groups and orders, Reprint of the 1981
  original, A Wiley-Interscience Publication.

\bibitem{Eis95}
D.~Eisenbud.
\newblock {\em Commutative algebra}, volume 150 of {\em Graduate Texts in
  Mathematics}.
\newblock Springer-Verlag, New York, 1995.
\newblock With a view toward algebraic geometry.

\bibitem{Gree10}
R.~Greenberg.
\newblock Surjectivity of the global-to-local map defining a {S}elmer group.
\newblock {\em Kyoto J. Math.}, 50(4):853--888, 2010.

\bibitem{Gree16}
R.~Greenberg.
\newblock On the structure of {S}elmer groups.
\newblock In {\em Elliptic curves, modular forms and {I}wasawa theory}, volume
  188 of {\em Springer Proc. Math. Stat.}, pages 225--252. Springer, Cham,
  2016.

\bibitem{Grei04}
C.~Greither.
\newblock Computing {F}itting ideals of {I}wasawa modules.
\newblock {\em Math. Z.}, 246(4):733--767, 2004.

\bibitem{GK08}
C.~Greither and M.~Kurihara.
\newblock Stickelberger elements, {F}itting ideals of class groups of
  {CM}-fields, and dualisation.
\newblock {\em Math. Z.}, 260(4):905--930, 2008.

\bibitem{GK15}
C.~Greither and M.~Kurihara.
\newblock Tate sequences and {F}itting ideals of {I}wasawa modules.
\newblock {\em Algebra i Analiz}, 27(6):117--149, 2015.

\bibitem{GK17}
C.~Greither and M.~Kurihara.
\newblock Fitting ideals of {I}wasawa modules and of the dual of class groups.
\newblock {\em Tokyo J. Math.}, 39(3):619--642, 2017.

\bibitem{GKT}
C.~Greither, M.~Kurihara, and H.~Tokio.
\newblock The second syzygy of the trivial {$G$}-module, and an equivariant
  main conjecture.
\newblock {\em To appear in the proceedings of Iwasawa 2017, Tokyo}.

\bibitem{GP12}
C.~Greither and C.~D. Popescu.
\newblock The {G}alois module structure of {$\ell$}-adic realizations of
  {P}icard 1-motives and applications.
\newblock {\em Int. Math. Res. Not. IMRN}, (5):986--1036, 2012.

\bibitem{GP15}
C.~Greither and C.~D. Popescu.
\newblock An equivariant main conjecture in {I}wasawa theory and applications.
\newblock {\em J. Algebraic Geom.}, 24(4):629--692, 2015.

\bibitem{Jan89}
U.~Jannsen.
\newblock Iwasawa modules up to isomorphism.
\newblock In {\em Algebraic number theory}, volume~17 of {\em Adv. Stud. Pure
  Math.}, pages 171--207. Academic Press, Boston, MA, 1989.

\bibitem{JN13}
H.~Johnston and A.~Nickel.
\newblock Noncommutative {F}itting invariants and improved annihilation
  results.
\newblock {\em J. Lond. Math. Soc. (2)}, 88(1):137--160, 2013.

\bibitem{Laz65}
M.~Lazard.
\newblock Groupes analytiques {$p$}-adiques.
\newblock {\em Inst. Hautes \'Etudes Sci. Publ. Math.}, (26):389--603, 1965.

\bibitem{NSW08}
J.~Neukirch, A.~Schmidt, and K.~Wingberg.
\newblock {\em Cohomology of number fields}, volume 323 of {\em Grundlehren der
  Mathematischen Wissenschaften [Fundamental Principles of Mathematical
  Sciences]}.
\newblock Springer-Verlag, Berlin, second edition, 2008.

\bibitem{Nic10}
A.~Nickel.
\newblock Non-commutative {F}itting invariants and annihilation of class
  groups.
\newblock {\em J. Algebra}, 323(10):2756--2778, 2010.

\bibitem{Nic13}
A.~Nickel.
\newblock Equivariant {I}wasawa theory and non-abelian {S}tark-type
  conjectures.
\newblock {\em Proc. Lond. Math. Soc. (3)}, 106(6):1223--1247, 2013.

\bibitem{Nic17}
A.~Nickel.
\newblock Notes on noncommutative {F}itting invariants.
\newblock {\em arXiv1712.07368}, 2017.

\bibitem{Nor76}
D.~G. Northcott.
\newblock {\em Finite free resolutions}.
\newblock Cambridge University Press, Cambridge-New York-Melbourne, 1976.
\newblock Cambridge Tracts in Mathematics, No. 71.

\bibitem{OV02}
Y.~Ochi and O.~Venjakob.
\newblock On the structure of {S}elmer groups over {$p$}-adic {L}ie extensions.
\newblock {\em J. Algebraic Geom.}, 11(3):547--580, 2002.

\bibitem{Per84}
B.~Perrin-Riou.
\newblock Arithm\'etique des courbes elliptiques et th\'eorie d'{I}wasawa.
\newblock {\em M\'em. Soc. Math. France (N.S.)}, (17):130, 1984.

\bibitem{Pop09}
C.~D. Popescu.
\newblock On the {C}oates-{S}innott conjecture.
\newblock {\em Math. Nachr.}, 282(10):1370--1390, 2009.

\bibitem{RW02}
J.~Ritter and A.~Weiss.
\newblock Toward equivariant {I}wasawa theory.
\newblock {\em Manuscripta Math.}, 109(2):131--146, 2002.

\bibitem{RW04}
J.~Ritter and A.~Weiss.
\newblock Toward equivariant {I}wasawa theory. {II}.
\newblock {\em Indag. Math. (N.S.)}, 15(4):549--572, 2004.

\bibitem{RW11}
J.~Ritter and A.~Weiss.
\newblock On the ``main conjecture'' of equivariant {I}wasawa theory.
\newblock {\em J. Amer. Math. Soc.}, 24(4):1015--1050, 2011.

\bibitem{Ros94}
J.~Rosenberg.
\newblock {\em Algebraic {$K$}-theory and its applications}, volume 147 of {\em
  Graduate Texts in Mathematics}.
\newblock Springer-Verlag, New York, 1994.

\bibitem{Ser65}
J.-P. Serre.
\newblock Sur la dimension cohomologique des groupes profinis.
\newblock {\em Topology}, 3:413--420, 1965.

\bibitem{Ven02}
O.~Venjakob.
\newblock On the structure theory of the {I}wasawa algebra of a {$p$}-adic
  {L}ie group.
\newblock {\em J. Eur. Math. Soc. (JEMS)}, 4(3):271--311, 2002.

\bibitem{Was97}
L.~C. Washington.
\newblock {\em Introduction to cyclotomic fields}, volume~83 of {\em Graduate
  Texts in Mathematics}.
\newblock Springer-Verlag, New York, second edition, 1997.

\bibitem{Wei13}
C.~A. Weibel.
\newblock {\em The {$K$}-book}, volume 145 of {\em Graduate Studies in
  Mathematics}.
\newblock American Mathematical Society, Providence, RI, 2013.
\newblock An introduction to algebraic $K$-theory.

\bibitem{Wil90}
A.~Wiles.
\newblock The {I}wasawa conjecture for totally real fields.
\newblock {\em Ann. of Math. (2)}, 131(3):493--540, 1990.

\end{thebibliography}
}

\end{document}